\documentclass{amsart}

\usepackage{amsfonts,amssymb,amsmath,amsthm}
\usepackage{url}
\usepackage{enumerate}

\numberwithin{equation}{section}

\newtheorem{theorem}{Theorem}[section]

\newtheorem{lemma}[theorem]{Lemma}
\newtheorem{proposition}[theorem]{Proposition}
\newtheorem{corollary}[theorem]{Corollary}
\theoremstyle{definition}

\theoremstyle{remark}
\newtheorem{remark}[theorem]{Remark}
\theoremstyle{problem}
\newtheorem{problem}[theorem]{Problem}

\newcommand{\qq}{{\mathbb Q}}
\newcommand{\rr}{{\mathbb R}}
\newcommand{\nn}{{\mathbb N}}
\newcommand{\zz}{{\mathbb Z}}

\DeclareMathOperator{\Res}{Res}
\DeclareMathOperator{\disc}{disc}

\date{November 24, 2017}

\author{Masanori~Sawa}
\address{
	Kobe University \\
	1-1 Rokkodai, Nada, Kobe, Hyogo 657-8501 \\
	Japan}
\email{sawa@people.kobe-u.ac.jp}

\author{Yukihiro~Uchida}
\address{Department of Mathematics and Information Sciences, Tokyo Metropolitan University,
         1-1 Minami-Osawa, Hachioji, Tokyo 192-0397 \\
         Japan}
\email{yuchida@tmu.ac.jp}

\thanks{The first author is supported in part by
Grant-in-Aid for Young Scientists (B) 26870259 and
Grant-in-Aid for Scientific Research (B) 15H03636
by the Japan Society for the Promotion of Science (JSPS).
The second author is also supported by
Grant-in-Aid for Young Scientists (B) 25800023 by JSPS}


\begin{document}

\title[
Discriminants of quasi-orthogonal polynomials, with applications]
{Discriminants of classical quasi-orthogonal polynomials, with combinatorial and number-theoretic applications}

\begin{abstract}
We derive explicit formulas for the resultants and discriminants of classical quasi-orthogonal polynomials,
as a full generalization of the results of Dilcher and Stolarsky (2005) and Gishe and Ismail (2008).
We consider a certain system of Diophantine equations,
originally designed by Hausdorff (1909) as a simplification of Hilbert's solution (1909)
of Waring's problem,
and then create the relationship to quadrature formulas and quasi-Hermite polynomials.
We reduce these equations to the existence problem of rational points
on a hyperelliptic curve associated with discriminants of
quasi-Hermite polynomials,
and thereby show a nonexistence theorem for solutions of Hausdorff-type equations.
\end{abstract}

\maketitle

\section{Introduction} \label{sec:Introduction}

The Jacobi, Laguerre and Hermite
polynomials are the classical orthogonal polynomials,
which, as we see in Szeg\H{o}'s book {\it Orthogonal Polynomials}, provide
a great deal of interesting topics in broad areas of mathematics.
We are here particularly concerned with a compact elegant formula for
the resultant and discriminant.

The resultant of two polynomials is a rather complicated function of their coefficients,
as pointed out by Dilcher and Stolarsky~\cite{DS05}.
Apostol~\cite{Ap75} found a general form of the resultant of two cyclotomic polynomials.
A quasi-orthogonal polynomial (of order one) is a polynomial of a sum of two orthogonal polynomials of consecutive degrees~\cite{X94}; Shohat and Tamarkin~\cite{ST70} uses the same term with a different naming.
Dilcher and Stolarsky~\cite{DS05} established a compact formula for
the resultant of two quasi-Chebyshev polynomials of the second kind.
Gishe and Ismail~\cite{GI08} (cf.~Gishe~\cite{G06})
obtained similar results concerning
quasi-Chebyshev polynomials of the first kind
and various ^^ Chebyshev-like' polynomials.

The discriminant of a polynomial is a special case of resultants.
Stieltjes~\cite{Sti88-1,Sti88-2} and Hilbert~\cite{Hi1888}
computed the discriminants of Jacobi, Laguerre and Hermite polynomials.
Dilcher and Stolarsky~\cite[Theorem~4]{DS05} derived
a compact formula for the discriminant of a quasi-Chebyshev polynomial of the second kind.

In this paper we derive explicit compact formulas for
the resultants and discriminants of all classical quasi-orthogonal polynomials,
as a full generalization of the results of Dilcher and Stolarsky, and Gishe and Ismail.
Our proof is a combination of algebraic properties on resultants and Schur's method based on the three-term relations of polynomials~\cite{Sch31} (see also~\cite[\S~6.71]{Sze}), which Dilcher and Stolarsky~\cite{DS05} and Gishe and Ismail~\cite{GI08} used to obtain their formulas; in addition, we employ various properties on classical orthogonal polynomials to derive explicit formulas for the discriminants.
We find a surprising connection with Hausdorff's work on Waring's problem. For this purpose of exploring this connection, we also consider a certain system of Diophantine equations,
originally designed by Hausdorff~\cite{Hau09}
(cf.~Nestarenko~\cite{N06})
as a simplification of Hilbert's solution~\cite{Hi09} of Waring's problem,
and elucidate the advantages of examining these equations through the discriminants of quasi-Hermite polynomials.

The paper is organized as follows.
Section~\ref{sec:Pre} gives preliminaries, where we review
some basic results on resultants and discriminants, quasi-orthogonal polynomials and quadrature formulas.
Sections~\ref{sec:resultant} through~\ref{sec:applications} are the main body of this paper.
In Section~\ref{sec:resultant}, we first establish a general formula for
the resultant of type $\Res(\Phi_n(x)+s\Phi_{n-1}(x),\Phi_{n-1}(x)+t\Phi_{n-2}(x))$,
where $\{\Phi_m\}$ is any sequence of orthogonal polynomials and the constants $s,t$ are arbitrarily chosen.
As a specialization of this result,
we obtain a compact formula for the resultants of all classical quasi-orthogonal polynomials.
Section~\ref{sec:explicit} is devoted to the discriminant of classical quasi-orthogonal polynomials. By employing the results in Section~\ref{sec:resultant}, we prove compact formulas for the discriminants of
quasi-Jacobi, quasi-Laguerre and quasi-Hermite polynomials.
In Section~\ref{sec:applications}, we give a generalization of Hausdorff's equations
and then show the relationship to quasi-Hermite polynomials and quadrature formulas for Gaussian integration.
We also show a nonexistence theorem for
solutions of such Hausdorff-type equations.
To do this, we reduce the problem to the existence of $\qq_2$-rational points
on a hyperelliptic curve associated with the discriminants of quasi-Hermite polynomials.
Section~\ref{sec:conclusion} is the conclusion, where further remarks will be made.

\section{Preliminary} \label{sec:Pre}

In this section we review some basic results on resultants and discriminants,
quasi-orthogonal polynomials and quadrature formulas.
We prove lemmas for further arguments in Sections~\ref{sec:resultant} through~\ref{sec:applications}.

\subsection{Discriminants, resultants and Schur's method} \label{subsec:Discriminat}

We first review resultants and discriminants.
For the details, we refer the reader to \cite[Chapter 12]{GKZ94}
or \cite[Chapter IV, \S~8]{L02}.

Let
\[
	f(x) = a_0 x^n + \dotsb + a_n, \quad
	g(x) = b_0 x^m + \dotsb + b_m
\]
be polynomials of degree $n$ and $m$ respectively.
The \emph{resultant} of $f$ and $g$ is defined by
\begin{equation}\label{eq:def-res}
	\Res(f, g) = \begin{vmatrix}
		a_0 & a_ 1 & \dots & a_n \\
		& \ddots & & & \ddots \\
		& & a_0 & a_ 1 & \dots & a_n \\
		b_0 & b_ 1 & \dots & b_m \\
		& \ddots & & & \ddots \\
		& & b_0 & b_ 1 & \dots & b_m
	\end{vmatrix},
\end{equation}
where the determinant is of order $(m+n)$.
Let $\alpha_1, \dotsc, \alpha_n$ be the zeros of $f(x)$ and
$\beta_1, \dotsc, \beta_m$ be the zeros of $g(x)$.
Then we have
\begin{equation}\label{eq:resultant}
	\Res(f, g) = a_0^m b_0^n \prod_{i=1}^n \prod_{j=1}^m
	(\alpha_i - \beta_j)
	= a_0^m \prod_{i=1}^n g(\alpha_i)
	= (-1)^{mn} b_0^n \prod_{j=1}^m f(\beta_j).
\end{equation}
The following properties immediately follow from \eqref{eq:resultant}.
\begin{align}
	\Res(f, g) &= (-1)^{mn} \Res(g, f), \label{eq:res-gf} \\
	\Res(f, g h) &= \Res(f, g) \Res(f, h), \label{eq:res-mult}
\end{align}
where $h$ is an arbitrary polynomial.

The following lemma, which was used by Dilcher and Stolarsky~\cite{DS05}, is useful to derive compact formulas for the resultants of quasi-orthogonal polynomials.
\begin{lemma}\label{lemma:Euclidean}
Let $f$ and $g$ be polynomials as above.
Let $q$ and $r$ be polynomials satisfying
\[
	f(x) = q(x) g(x) + r(x).
\]
Then we have
\begin{equation}\label{eq:Euclidean}
	\Res(g, f) = b_0^{\deg f-\deg r}\Res(g, r).
\end{equation}
\end{lemma}
\begin{proof}[Proof of Lemma~\ref{lemma:Euclidean}]
The lemma follows from \eqref{eq:resultant}.
See also \cite[Lemma~4.1]{DS05}.
\end{proof}

The \emph{discriminant} of $f$ is defined by
\begin{equation}\label{eq:def-disc}
	\disc(f) = a_0^{2n-2} \prod_{1\le i<j\le n} (\alpha_i - \alpha_j)^2.
\end{equation}
The discriminant of $f$ is represented in terms of a resultant as follows.
\begin{equation}\label{eq:DR}
	\disc(f) = \frac{(-1)^{n(n-1)/2}}{a_0} \Res(f, f').
\end{equation}

\begin{remark}
The sign in the right-hand side differs in some literature.
For example, the sign $(-1)^{n(n-1)/2}$ does not appear in~\cite{GKZ94}
because the definition of discriminants differs by sign.
\end{remark}

The following proposition follows
from \eqref{eq:def-res} and \eqref{eq:DR}.
\begin{proposition}\label{prop:homogeneity}
Let $f(x)=a_0x^n+\dotsb+a_n$. Then $\disc(f)$ is a homogeneous polynomial in $a_0,\dotsc,a_n$ of degree $2n-2$ with integer coefficients.
\end{proposition}

By Proposition~\ref{prop:homogeneity}, we may substitute a polynomial of degree less than $n$ for $f$ in $\disc(f)$.
If necessary, we use the notation $\disc_n(f)$ to emphasize the dependence on $n$.

\begin{proposition}\label{prop:degree}
Let $f(x)=a_0x^n+\dotsb+a_n$. Then we have
\[
	\disc_{n+1}(f) = a_0^2 \disc_n(f).
\]
\end{proposition}
\begin{proof}[Proof of Proposition~\ref{prop:degree}]
See~\cite[Chapter~12, (1.41)]{GKZ94}.
Note that the definition of discriminants in~\cite{GKZ94} differs by sign from ours.
\end{proof}

\begin{proposition}
\label{prop:DS4.3}
Let $f(x)=a_0x^n+\dotsb+a_n$ and let $a,b,c$ be constants.
Then we have
\begin{align*}
	\disc(f(ax+b)) &= a^{n(n-1)} \disc(f(x)), \\
	\disc(cf(x)) &= c^{2(n-1)} \disc(f(x)).
\end{align*}
\end{proposition}
\begin{proof}[Proof of Proposition~\ref{prop:DS4.3}]
The proposition follows from~\eqref{eq:def-disc}.
See also~\cite[Lemma~4.3]{DS05}.
\end{proof}

\begin{proposition}\label{prop:DLC}
Let $p(x)$ and $q(x)$ be polynomials of degree $n$ and $n-1$ respectively.
Let $c$ be a constant.
\begin{enumerate}
\item[{\rm (i)}] \label{item:DLC-poly} 
The discriminant $\disc(p+cq)$ is a polynomial in $c$ and
\[
	\deg \disc(p+cq) \le 2 (n - 1).
\]
The equality holds if and only if $q$ has no multiple zeros.
\item[{\rm (ii)}] \label{item:DLC-even} 
If $p(-x)=(-1)^n p(x)$ and $q(-x)=(-1)^{n-1}q(x)$,
then $\disc(p+cq)$ is an even polynomial in $c$.
\end{enumerate}
\end{proposition}
\begin{proof}[Proof of Proposition~\ref{prop:DLC}]
(i) By Proposition~\ref{prop:homogeneity},
$\disc(p+cq)$ is a polynomial in $c$.
By Proposition~\ref{prop:DS4.3}, we have
\[
	\frac{\disc(p+cq)}{c^{2(n-1)}} = \disc_n
	\left(\frac{1}{c}p+q\right).
\]
By Proposition~\ref{prop:degree}, we have
\[
	\lim_{c\to\infty} \frac{\disc(p+cq)}{c^{2(n-1)}}
	= \disc_n(q)
	= l^2 \disc_{n-1}(q),
\]
where $l$ is the leading coefficient of $q$.
This completes the proof.

\noindent
(ii)
By assumption,
\[
	p(-x) + c q(-x)
	= (-1)^n (p(x) - c q(x)).
\]
Therefore, by Proposition~\ref{prop:DS4.3},
\begin{align*}
	\disc(p - c q) &= \disc((-1)^n (p(-x) + c q(-x))) \\
	&= (-1)^{n\cdot 2(n-1)}(-1)^{n(n-1)}\disc(p(x) + c q(x)) \\
	&= \disc(p + c q). \qedhere
\end{align*}
\end{proof}

The following lemma, due to Schur~\cite{Sch31} (see~\cite[\S~6.71]{Sze}),
plays a role in the proof of the main theorems of
Sections~\ref{sec:resultant} and~\ref{sec:explicit}.

\begin{lemma}
[Schur's method]
\label{lemma:delta}
Let $\{\Phi_m\}$ be a sequence of polynomials satisfying
\begin{equation}\label{eq:rho}
\begin{split}
	\Phi_m(x) &= (a_m x + b_m) \Phi_{m-1}(x) - c_m \Phi_{m-2}(x), \\
	\Phi_0(x) &= 1, \quad \Phi_1(x) = a_1 x + b_1,
\end{split}
\end{equation}
where $a_m$, $b_m$, $c_m$ are constants with $a_m c_m \neq 0$.
Let $y_1, \dotsc, y_n$ be the zeros of $\Phi_n(x)$.
Then we have
\[
   \prod_{k=1}^n \Phi_{n-1}(y_k) = (-1)^{n(n-1)/2} \prod_{k=1}^n a_k^{n-2k+1} c_k^{k-1}.
\]
\end{lemma}

\subsection{Quasi-orthogonal polynomials and Riesz-Shohat Theorem} \label{subsec:Quasi}

Let $\mu$ be a positive Borel measure on an interval $(a,b)$ with finite moments.
For convenience, we assume that $\int_a^b d\mu=1$.
Let $\{ \Phi_n \}$ be a sequence of \emph{orthogonal polynomials} with respect to $\mu$, namely
\[
	\int_a^b \Phi_m(x) \Phi_n(x) d\mu = 0 \quad \text{for every distinct $m$ and $n$}.
\]

By Bochner's theorem,
the classical orthogonal polynomials are completely classified by
the \emph{Jacobi polynomials}, \emph{Laguerre polynomials}, \emph{Hermite polynomials}.

\bigskip

\noindent
\textbf{Jacobi polynomial.} 
For $\alpha,\beta>-1$, 
the $n$-th Jacobi polynomial $P^{(\alpha,\beta)}_n(x)$ is defined
by the Rodrigues' formula as follows:
\begin{equation}\label{eq:Rod-J}
	P^{(\alpha,\beta)}_n(x) = \frac{(-1)^n}{2^n n!}
	(1 - x)^{-\alpha} (1 + x)^{-\beta}
	\frac{d^n}{dx^n} \left( (1 - x)^{n+\alpha} (1 + x)^{n + \beta} \right).
\end{equation}
The polynomials $P^{(\alpha,\beta)}_n(x)$ are orthogonal with respect to $(1-x)^\alpha (1+x)^\beta$ on $(-1,1)$.

\bigskip

\noindent
\textbf{Laguerre polynomial.} 
For $\alpha>-1$,
the $n$-th Laguerre polynomial $L_n^{(\alpha)}(x)$ is defined
by the Rodrigues' formula as follows:
\begin{equation}\label{eq:Rod-L}
	L_n^{(\alpha)}(x) =
	\frac{e^x x^{-\alpha}}{n!} \frac{d^n}{dx^n} (e^{-x} x^{n+\alpha}).
\end{equation}
The polynomials $L_n^{(\alpha)}(x)$ are orthogonal with respect to $e^{-x}x^{\alpha}$ on $(0,\infty)$.

\bigskip

\noindent
\textbf{Hermite polynomial.} 
The $n$-th Hermite polynomial $H_n(x)$ is defined
by the Rodrigues' formula as follows:
\begin{equation}\label{eq:Rod-H}
	H_n(x) = (-1)^n e^{x^2} \frac{d^n}{dx^n} e^{-x^2}.
\end{equation}
The polynomials $H_n(x)$ are orthogonal with respect to $e^{-x^2}$ on $\rr$.

\bigskip

Some of the basic properties on classical orthogonal polynomials,
used in
Sections~\ref{sec:resultant} through~\ref{sec:applications},
are summarized in Appendix~\ref{appendix:COP}.
For the general theory, we refer the readers to Szeg\H{o}'s book
{\it Orthogonal Polynomials}~\cite[Chapter IV and \S~5.1 and \S~5.5]{Sze}.

A {\it quasi-orthogonal polynomial of degree $n$ and order $r$} is
a polynomial of type
\[
	\Phi_{n,r}(x) = \Phi_n(x) + b_1 \Phi_{n-1}(x) + \dotsb + b_r \Phi_{n-r}(x)
\]
in which $b_1,\ldots,b_r \in \rr$ and $b_r \ne 0$~\cite{X94}.
For convenience, we set $\Phi_{n,0}(x) = \Phi_n(x)$.
The polynomial $\Phi_{n,r}(x)$ is orthogonal to all polynomials of degree at most $n-r-1$.

\begin{remark}
\label{rem:QOP}
Shohat and Tamarkin~\cite{ST70} used the term
^^ order' of $\Phi_{n,r}(x)$, with a different meaning.
\end{remark}

The following is also well known (cf.~\cite[Theorem~3.3.4]{Sze}):

\begin{proposition}\label{prop:zero3}
Let $b_1$ be a real constant.
Then the polynomial $\Phi_{n+1,1}(x) = \Phi_{n+1}(x) + b_1 \Phi_n(x)$
has $n+1$ distinct real roots.
\end{proposition}

The following result was first obtained by Riesz~\cite[p.23]{Ri23} for $k=2$,
and generalized by Shohat~~\cite[Theorem~I]{Sho37} for $k \ge 3$.

\begin{theorem}
[Riesz-Shohat Theorem]
\label{thm:Shohat}
Let $c_1,\dotsc,c_n$ be distinct real numbers,
$\omega_n(x)=\prod_{i=1}^n(x-c_i)$ and
\[
	\gamma_i = \int_a^b \frac{\omega_n(x)}{(x-c_i)\omega_n'(c_i)} d\mu.
\]
Let $k$ be an integer with $1\le k\le n+1$.
The following are equivalent.
\begin{enumerate}
\item[{\rm (i)}]
The equation
\begin{equation}
\label{eq:quadrature2}
	\sum_{i=1}^n \gamma_i f(c_i) = \int_a^b f(x) d\mu,
\end{equation}
holds for all polynomials $f(x)$ of degree at most $2n-k$.
\item[{\rm (ii)}]
For all polynomials $g(x)$ of degree at most $n-k$,
\[
	\int_a^b \omega_n(x) g(x) d\mu = 0.
\]
\item[{\rm (iii)}]
The polynomial $\omega_n(x)$ is
a quasi-orthogonal polynomial of degree $n$ and order $k-1$,
that is, there exists real numbers $b_1,\dotsc,b_{k-1}$ such that
\[
	\omega_n(x) = \Phi_n(x) + b_1 \Phi_{n-1}(x) + \dotsb + b_{k-1} \Phi_{n-k+1}(x).
\]
\end{enumerate}
\end{theorem}

Integration formulas of type (\ref{eq:quadrature2}) are called \emph{quadrature formulas}.
Quadrature formulas with positive weights $\gamma_i$ are important as integration
formula, which, by a theorem of Xu~\cite[Theorem 4.1]{X94},
have an elegant characterization in terms of tri-diagonal matrices.
A class of positive quadrature formulas was also implicitly used in Hausdorff's work~\cite{Hi09} on Waring's problem; the details will be clear in the next subsection.

\subsection{Quadrature formulas and Hausdorff's Lemma} \label{subsec:Hi-qf}

A {\it Hilbert identity} is a polynomial identity of the form
\begin{equation}\label{eq:RHI}
(x_1^2 + \cdots + x_n^2)^r
= \sum_{i = 1}^M c_i (a_{i1} x_1 + \cdots + a_{in} x_n)^{2r}
\end{equation}
where $0 < c_i$ and $a_{ij} \in \rr$.
Clearly, it is always possible to absorb the coefficients $c_i$'s into the linear forms.
A {\it rational identity} is an identity of type (\ref{eq:RHI})
in which $0 < c_i \in \qq$ and $a_{ij} \in \qq$.
In this case scaling is no longer simple.

Waring's problem in number theory asks whether
every positive integer can be expressed as a sum of $r$-th powers of integers.
The case $r=2$ had been stated by Fermat in 1640 and was solved by Lagrange in 1770.
The first advance for $r \geq 3$ was made by Liouville in 1859,
who proved that 
every natural integer is a sum of at most $53$ fourth powers of integers.
For this purpose, Liouville used the rational identity
\[
6(x_1^2+x_2^2+x_3^2+x_4^2)^2
= \sum_{1 \leq i < j \leq 4} \left\{ (x_i+x_j)^4 + (x_i-x_j)^4 \right\}.
\]
Mathematicians in the rest of the 19th century gave
similar identities and settled Waring's problem in the small-degree cases.
For a good introduction to the early histories on Waring's problem,
we refer the readers to Dickson's book
{\em History of the Theory of Numbers, II}~\cite[pp.717-725]{D23}.

It was Hilbert~\cite{Hi09} who finally solved
Waring's problem in general; namely, for every positive integer $r$, there
exists some positive integer $g(r)$ so that for each $n \in \nn$ there exist $x_k \in \zz$
so that
\[ n = \sum_{k=1}^{g(r)} x_k^r. \]
We are concerned here only with the first part of Hilbert's proof,
which involved the construction of rational Hilbert identities.

The first key step of Hilbert's proof is Theorem~\ref{thm:Hilbert} below, which
was stated for $n=5$; it is obvious that Hilbert's argument applies to general values of $n$.

\begin{theorem}
[Hilbert's Lemma]
\label{thm:Hilbert}
For every positive integers $n$ and $r$,
\[
  (x_1^2 + \cdots + x_n^2)^r
     = \sum_{i = 1}^M c_i (a_{i1} x_1 + \cdots + a_{in} x_n)^{2r}
\]
in which $M = (2r+n-1) \cdots (2r+1)/(n-1)!$, $0 < c_i \in \qq$ and $a_{ij} \in \qq$.
\end{theorem}

Hilbert found his identities in two steps.
First, he showed that if $d\mu$ is a suitably-normalized surface measure on $S^{n-1}$
and $x_i$'s are taken parameters, then
\begin{equation}
\label{eq:int-sphere}
{\int\cdots \int}_{u \in S^{n-1}} (x_1u_1 + \cdots + x_nu_n)^{2r} d\mu = (x_1^2 + \cdots + x_n^2)^r.
\end{equation}
By approximating the integral with a Riemann sum and then using some elementary arguments,
he derived the existence of real Hilbert identities.
Then by a standard continuity argument, Hilbert found his rational identities. 
There have been some expository works which, while mainly concerned with Waring's problem,
described Hilbert's Theorem. For the details, we refer the readers to Pollack~\cite{P09}.

The first simplification of Hilbert's result was made by Hausdorff \cite{Hau09},
who replaced the integral on the left of (\ref{eq:int-sphere}) by
the Gaussian integral
\[{\int\cdots \int}_{u \in \rr^n}e^{-(u_1^2 + \cdots + u_n^2)} (x_1u_1 + \cdots + x_nu_n)^{2r}\ du_1 \cdots du_n,\]
and showed that, up to a constant, the value is $(\sum x_i^2)^r$ again.
Then he constructed an iterated sum which leads to explicit real Hilbert identities
in any number of variables, by using the roots of the Hermite polynomial $H_{2r}$
and then showing the following key lemma:
\begin{theorem}[Hausdorff's Lemma]
\label{lem:haus}
Let $r$ be a positive integer. Then there exist rationals $x_1, \ldots, x_{r+1}$, $y_1, \ldots, y_{r+1}$ such that
\begin{equation}
\label{eq:hausdorff-equation1}
\sum_{i=1}^{r+1} x_i y_i^j = \frac{1}{\sqrt{\pi}} \int_{-\infty}^\infty t^j \ e^{-t^2}dt, \qquad j = 0, 1, \ldots, r. 
\end{equation}
\end{theorem}
Hausdorff then quickly argued that the real coefficients may be replaced by rational ones,
and completed another proof of Hilbert's Lemma.
For example, see Nesterenko~\cite{N06} for more details and further refinements to Hausdorff's result.

Diophantine equations of type (\ref{eq:hausdorff-equation1}) are important as quadrature formulas.
Let $\xi$ be a positive Borel measure on an interval $(a,b)$.
Let $x_1, \ldots, x_m \in \rr$ and $y_1, \ldots, y_m \in (a,b)$.
A {\it quadrature formula of degree $t$} is an integration formula of type
\begin{equation} \label{eq:cuba}
\sum_{i=1}^m x_i f(y_i) = \int_a^b f(x) \ d\xi
\end{equation}
in which $f$ ranges over all polynomials of degree at most $t$.
The points $y_i$ are called {\it nodes} and coefficients $x_i$ are called {\it weights}.
A quadrature formula is {\it positive} if all weights are positive.
We see that the equations (\ref{eq:hausdorff-equation1}) are equivalent to
a {\it rational quadrature}, meaning a quadrature formula of degree $r$ for Gaussian integration
$\frac{1}{\sqrt{\pi}} \int_{-\infty}^\infty \ e^{-t^2}dt$
with rational nodes and weights.
In Subsection~\ref{subsec:HDeq},
we formulate Diophantine equations of type (\ref{eq:hausdorff-equation1}) in general.

The concept of quadrature formula is simply generalized to higher dimensions
and integrands may be also replaced by the homogeneous polynomials.
A {\it cubature formula of index $t$} is an integration formula of type (\ref{eq:cuba})
in which $f$ ranges over all homogeneous polynomials of degree $t$.
The relationship of Hilbert identities to index-type cubature formulas for $\int_{S^{n-1}} \ d\rho$,
where $\rho$ is a surface measure on $S^{n-1}$,
goes back to the 19th century at least~\cite{Ry70}.
Interest was revived in the development of \emph{spherical designs}
by Delsarte, Goethals and Seidel in the 1970s~\cite{DGS77}.
By a suitable scaling of weights and nodes,
cubature formulas for $\int_{S^{n-1}} \ d\rho$ and
${\int \cdots \int}_{\rr^n} \ e^{-(u_1^2 + \cdots + u_n^2)} du_1 \cdots du_n$
can be transformed to each other (cf.~\cite{BB04,NS14}).
We can easily construct a cubature formula for Gaussian integration
by taking copies of a quadrature formula for $\frac{1}{\sqrt{\pi}} \int_{-\infty}^\infty \ e^{-t^2} dt$
and then taking their convolutions.
This is an example of the widely-used method in the study of cubature formulas,
called {\it product construction}~\cite{S71},
and explains why Hausdorff's simplification works well.

\section{Compact formulas for resultants of classical quasi-orthogonal polynomials} \label{sec:resultant}

In this section, we first establish a general formula for the resultant of type $\Res(\Phi_n(x)+s\Phi_{n-1}(x),\Phi_{n-1}(x)+t\Phi_{n-2}(x))$, where $\{\Phi_m\}$ is any sequence of orthogonal polynomials and the constants $s,t$ are arbitrarily chosen.
We then derive, as a specialization of this result, compact formulas for the resultants of all classical quasi-orthogonal polynomials.

\begin{theorem}
\label{thm:res-general}
Let $\{\Phi_m\}$ be a sequence of polynomials satisfying
\begin{equation}\label{eq:rho2}
\begin{split}
	\Phi_m(x) &= (a_m x + b_m) \Phi_{m-1}(x) - c_m \Phi_{m-2}(x), \\
	\Phi_0(x) &= 1, \quad \Phi_1(x) = a_1 x + b_1,
\end{split}
\end{equation}
where $a_m$, $b_m$, $c_m$ are constants with $a_m c_m \neq 0$.
Let $\Phi_{m;c}(x)=\Phi_{m}(x)+c\Phi_{m-1}(x)$ for a constant $c$.
Let $n\ge 2$ be an integer and let $s$ and $t$ be constants.
Then
\begin{equation} \label{eq:res-general}
\begin{split}
	\Res(\Phi_{n;s}, \Phi_{n-1;t})
	&= (-1)^{n(n+1)/2} a_n^n c_n^{-n}
	\prod_{k=1}^{n} a_k^{2n-2k-1} c_k^{k-1} \\
	&\quad \cdot t^n
\Phi_{n;s}\left(-\frac{c_n + b_n t + s t}{a_n t}\right).
\end{split}
\end{equation}
In particular, if $t=0$, then
\begin{equation} \label{eq:res-general-t0}
	\Res(\Phi_{n;s}, \Phi_{n-1})
	= \Res(\Phi_n, \Phi_{n-1})
	= (-1)^{n(n-1)/2} \prod_{k=1}^{n} a_k^{2n-2k} c_k^{k-1}.
\end{equation}
\end{theorem}
\begin{proof}[Proof of Theorem~\ref{thm:res-general}]
We first prove \eqref{eq:res-general-t0}.
Let $l_n$ be the leading coefficient of $\Phi_n(x)$
and let $y_1, \dotsc, y_n$ be the zeros of $\Phi_n(x)$.
By Lemma~\ref{lemma:Euclidean} and \eqref{eq:resultant},
\[
	\Res(\Phi_{n;s}, \Phi_{n-1})
	= \Res(\Phi_n, \Phi_{n-1})
	= l_n^{n-1} \prod_{i=1}^n \Phi_{n-1}(y_i).
\]
We have $l_n=\prod_{k=1}^n a_k$ by \eqref{eq:rho2}.
Hence, by Lemma~\ref{lemma:delta},
\[
	l_n^{n-1} \prod_{i=1}^n \Phi_{n-1}(y_i)
	= (-1)^{n(n-1)/2} \prod_{k=1}^n a_k^{2n-2k} c_k^{k-1}.
\]

Next, we prove \eqref{eq:res-general}.
By \eqref{eq:rho2},
\begin{align*}
	\Phi_{n-1;t}(x) &= \Phi_{n-1}(x) + t \Phi_{n-2}(x) \\
	&= - c_n^{-1} t \Phi_n(x)
	+ (1 + c_n^{-1} t (a_n x + b_n)) \Phi_{n-1}(x) \\
	&= - c_n^{-1} t \Phi_{n;s}(x)
	+ c_n^{-1} (a_n t x + c_n + b_n t + s t) \Phi_{n-1}(x).
\end{align*}
Hence, by Lemma~\ref{lemma:Euclidean} and \eqref{eq:res-mult},
\begin{align*}
	&\Res(\Phi_{n;s}, \Phi_{n-1;t}) \\
	&= \Res(\Phi_{n;s}(x), - c_n^{-1} t \Phi_{n;s}(x)
	+ c_n^{-1} (a_n t x + c_n + b_n t + s t) \Phi_{n-1}(x)) \\
	&= l_n^{-1} \Res(\Phi_{n;s}(x),
	c_n^{-1} (a_n t x + c_n + b_n t + s t) \Phi_{n-1}(x)) \\
	&= l_n^{-1} \Res(\Phi_{n;s},
	c_n^{-1} (a_n t x + c_n + b_n t + s t))
	\Res(\Phi_{n;s}, \Phi_{n-1}).
\end{align*}
By \eqref{eq:resultant},
\[
	\Res(\Phi_{n;s}, c_n^{-1} (a_n t x + c_n + b_n t + s t))
	= (-1)^n \left(\frac{a_n t}{c_n}\right)^n
	\Phi_{n;s}\left(-\frac{c_n + b_n t + s t}{a_n t}\right).
\]
Therefore, by \eqref{eq:res-general-t0},
\begin{align*}
	\Res(\Phi_{n;s}, \Phi_{n-1;t})
	&= (-1)^{n(n+1)/2} a_n^n c_n^{-n}
	\prod_{k=1}^{n} a_k^{2n-2k-1} c_k^{k-1} \\
	&\quad \cdot t^n \Phi_{n;s}\left(-\frac{c_n + b_n t + s t}{a_n t}\right). \qedhere
\end{align*}
\end{proof}

\subsection{Classical quasi-orthogonal polynomials} \label{subsec:res-QJ}

We here describe explicit formulas for the resultants of classical quasi-orthogonal polynomials.

\begin{theorem} \label{thm:res-QJ}
Let $P_{n;c}^{(\alpha,\beta)}(x)
=P_n^{(\alpha,\beta)}(x)+cP_{n-1}^{(\alpha,\beta)}(x)$
for a constant $c$.
Let $n\ge 2$ be an integer and let $s$ and $t$ be constants.
Then
\begin{equation} \label{eq:res-QJ}
\begin{split}
	&\Res(P^{(\alpha,\beta)}_{n;s}, P^{(\alpha,\beta)}_{n-1;t}) \\
	&= \frac{(-1)^{n(n+1)/2} (2n+\alpha+\beta)^{n-2}
	(2n+\alpha+\beta-1)^n (2n+\alpha+\beta-2)^n}
	{2^{n(n-1)} (n+\alpha-1)^n (n+\beta-1)^n} \\
	&\quad \cdot \prod_{k=1}^n k^{k-2n+2}
	\prod_{k=1}^{n-1} (k+\alpha)^k (k+\beta)^k (n+k+\alpha+\beta)^{n-k-2} \\
	&\quad \cdot t^n P^{(\alpha,\beta)}_{n;s}
	\biggl(-\frac{2 (n + \alpha - 1) (n + \beta - 1)}
	{(2n + \alpha + \beta - 1) (2n + \alpha + \beta - 2) t} \\
	&\quad - \frac{\alpha^2 - \beta^2}
	{(2n + \alpha + \beta) (2n + \alpha + \beta - 2)}
	- \frac{2n (n + \alpha + \beta) s}
	{(2n + \alpha + \beta) (2n + \alpha + \beta - 1)}\biggr).
\end{split}
\end{equation}
In particular, if $t=0$, then
\begin{equation} \label{eq:res-QJ-t0}
\begin{split}
	\Res(P^{(\alpha,\beta)}_{n;s}, P^{(\alpha,\beta)}_{n-1})
	&= \frac{(-1)^{n(n-1)/2} (2n+\alpha+\beta)^{n-1}}
	{2^{n(n-1)}} \prod_{k=1}^n k^{k-2n+1} \\
	&\quad \cdot \prod_{k=1}^{n-1}
	(k+\alpha)^k (k+\beta)^k (n+k+\alpha+\beta)^{n-k-1}.
\end{split}
\end{equation}
\end{theorem}
\begin{proof}[Proof of Theorem~\ref{thm:res-QJ}]
By \eqref{eq:J-rec}, the sequence $\{P^{(\alpha,\beta)}_m\}$ satisfies
\eqref{eq:rho2} for
\[
\begin{gathered}
	a_m = \frac{(2m + \alpha + \beta - 1) (2m + \alpha + \beta)}
	{2m (m + \alpha + \beta)}, \\
	b_m = \frac{(2m + \alpha + \beta - 1) (\alpha^2 - \beta^2)}
	{2m (m + \alpha + \beta) (2m + \alpha + \beta - 2)}, \quad
	c_m = \frac{(m + \alpha - 1) (m + \beta - 1) (2m + \alpha + \beta)}
	{m (m + \alpha + \beta) (2m + \alpha + \beta - 2)}.
\end{gathered}
\]
Since
\begin{align*}
	&\prod_{k=1}^n a_k^{2n-2k-1} c_k^{k-1} \\
	&= \prod_{k=1}^n 2^{2k-2n+1} k^{k-2n+2} (k+\alpha+\beta)^{k-2n+2}
	(2k+\alpha+\beta-1)^{2n-2k-1} \\
	&\quad \cdot\prod_{k=1}^{n-1}
	(k+\alpha)^k (k+\beta)^k (2k+\alpha+\beta)^{2n-2k-2} \\
	&= 2^{-n(n-2)} (2n+\alpha+\beta)^{n-2} \prod_{k=1}^{n} k^{k-2n+2} \\
	&\quad \cdot \prod_{k=1}^{n-1}
	(k+\alpha)^k (k+\beta)^k (n+k+\alpha+\beta)^{n-k-2},
\end{align*}
we obtain \eqref{eq:res-QJ} by Theorem~\ref{thm:res-general}.
The proof of \eqref{eq:res-QJ-t0} is similar.
\end{proof}

The $n$-th \emph{Chebyshev polynomial of the first kind} is defined by
\begin{equation}
\label{eq:Chebyshev_first}
	T_n(x) = \frac{P^{(-1/2,-1/2)}_n(x)}{P^{(-1/2,-1/2)}_n(1)}
	= \binom{n-\frac{1}{2}}{n}^{-1} P^{(-1/2,-1/2)}_n(x).
\end{equation}
When $n\ge 1$, we have
\begin{equation}
\label{eq:C1G}
	T_n(x) = \lim_{\lambda\to 0} \frac{n}{2\lambda} C^{(\lambda)}_n(x).
\end{equation}
The $n$-th \emph{Chebyshev polynomial of the second kind} is defined by
\begin{equation}
\label{eq:Chebyshev_second}
	U_n(x) = (n+1)\frac{P^{(1/2,1/2)}_n(x)}{P^{(1/2,1/2)}_n(1)}
	= C^{(1)}_n(x).
\end{equation}

\begin{corollary} \label{cor:res-QC}
For a constant $c$, let $T_{n;c}(x)=T_n(x)+cT_{n-1}(x)$
and $U_{n;c}(x)=U_n(x)+cU_{n-1}(x)$.
Let $n\ge 2$ be an integer and let $s$ and $t$ be constants.
Then
\begin{align}
	\label{eq:res-T}
	\Res(T_{n;s}, T_{n-1;t})
	&= (-1)^{n(n+1)/2} 2^{n^2-3n+3} t^n
	T_{n;s}\left(-\frac{1+st}{2t}\right), \\
	\label{eq:res-U}
	\Res(U_{n;s}, U_{n-1;t})
	&= (-1)^{n(n+1)/2} 2^{n(n-1)} t^n
	U_{n;s}\left(-\frac{1+st}{2t}\right).
\end{align}
\end{corollary}
\begin{proof}[Proof of Corollary~\ref{cor:res-QC}]
The sequence $\{T_m\}$ satisfies
\[
	T_m(x) = 2 x T_{m-1}(x) - T_{m-2}(x), \quad
	T_0(x) = 1, \quad T_1(x) = x.
\]
In other words, $\{T_m\}$ satisfies \eqref{eq:rho2}
for $a_1=1$, $a_m=2$ for $m\ge 2$,
and $b_m=0$ and $c_m=1$ for $m\ge 1$.
Then
\[
	a_n^n c_n^{-n} \prod_{k=1}^n a_k^{2n-2k-1} c_k^{k-1}
	= 2^n \prod_{k=2}^n 2^{2n-2k-1} = 2^{n^2-3n+3}.
\]
Therefore we obtain \eqref{eq:res-T} by Theorem~\ref{thm:res-general}.

Similarly, the sequence $\{U_m\}$ satisfies
\[
	U_m(x) = 2 x U_{m-1}(x) - U_{m-2}(x), \quad
	U_0(x) = 1, \quad U_1(x) = 2 x,
\]
that is, $\{U_m\}$ satisfies \eqref{eq:rho2} for
$a_m=2$, $b_m=0$ and $c_m=1$ for $m\ge 1$.
Then
\[
	a_n^n c_n^{-n} \prod_{k=1}^n a_k^{2n-2k-1} c_k^{k-1}
	= 2^n \prod_{k=1}^n 2^{2n-2k-1} = 2^{n(n-1)}.
\]
Therefore we obtain \eqref{eq:res-U} by Theorem~\ref{thm:res-general}.
\end{proof}

\begin{remark}
Dilcher and Stolarsky~\cite[Theorem~2]{DS05} derived the formula
\begin{equation} \label{eq:DS-Theorem2}
\begin{split}
	&\Res(U_{n;s}, U_{n-1;t}) \\
	&= (-1)^{n(n-1)/2} 2^{n(n-1)}
	t^n \left(U_n\left(\frac{1+st}{2t}\right)
	- s U_{n-1}\left(\frac{1+st}{2t}\right)\right).
\end{split}
\end{equation}
The equivalence of \eqref{eq:res-U} and \eqref{eq:DS-Theorem2}
is easily seen since $U_m(-x)=(-1)^mU_m(x)$.
Gishe and Ismail~\cite[Theorem~2.1]{GI08} gave another proof
of \eqref{eq:DS-Theorem2} by using Schur's method (Lemma~\ref{lemma:delta}).
They also derived a formula equivalent to \eqref{eq:res-T}
(see~\cite[Theorem~3.1]{GI08}).
\end{remark}

Next, we describe explicit formulas for the resultants of quasi-Laguerre and quasi-Hermite polynomials.

\begin{theorem}\label{thm:res-QL}
Let $L_{n;c}^{(\alpha)}(x)=L_n^{(\alpha)}(x)+cL_{n-1}^{(\alpha)}(x)$
for a constant $c$.
Let $n\ge 2$ be an integer and let $s$ and $t$ be constants.
Then
\begin{equation} \label{eq:res-L-compact}
\begin{split}
	\Res(L^{(\alpha)}_{n;s}, L^{(\alpha)}_{n-1;t})
	&= \frac{(-1)^{n(n+1)/2}}{(n+\alpha-1)^n}
	\prod_{k=1}^{n} k^{k-2n+2}
	\prod_{k=1}^{n-1} (k+\alpha)^k \\
	& \quad \cdot t^n L^{(\alpha)}_{n;s}\left(
	\frac{n+\alpha-1+(2n+\alpha-1)t+nst}{t}\right).
\end{split}
\end{equation}
\end{theorem}
\begin{proof}[Proof of Theorem~\ref{thm:res-QL}]
By \eqref{eq:L-rec}, the sequence $\{L^{(\alpha)}_m\}$ satisfies
\eqref{eq:rho2} for
\[
	a_m = -\frac{1}{m}, \quad
	b_m = \frac{2m + \alpha - 1}{m}, \quad
	c_m = \frac{m + \alpha - 1}{m}.
\]
Since
\begin{align*}
	\prod_{k=1}^n a_k^{2n-2k-1} c_k^{k-1}
	&= \prod_{k=1}^n \frac{(-1)^{2n-2k-1}(k+\alpha-1)^{k-1}}{k^{2n-k-2}} 
        = (-1)^n \prod_{k=1}^{n} k^{k-2n+2}
	\prod_{k=1}^{n-1} (k+\alpha)^k,
\end{align*}
we obtain the theorem by Theorem~\ref{thm:res-general}.
\end{proof}

\begin{theorem}\label{thm:res-QH}
Let $H_{n;c}(x)=H_n(x)+cH_{n-1}(x)$ for a constant $c$.
Let $n\ge 2$ be an integer and let $s$ and $t$ be constants.
Then
\begin{equation} \label{eq:res-H-compact}
	\Res(H_{n;s}, H_{n-1;t}) = 
	\frac{(-1)^{n(n+1)/2} 2^{n(3n-5)/2}}{(n-1)^n} \prod_{k=1}^{n-1} k^k
	\cdot t^n H_{n;s}\left(-\frac{2(n-1)+st}{2t}\right).
\end{equation}
\end{theorem}
\begin{proof}[Proof of Theorem~\ref{thm:res-QH}]
By \eqref{eq:H-rec}, the sequence $\{H_m\}$ satisfies
\eqref{eq:rho2} for
\[
	a_m = 2, \quad
	b_m = 0, \quad
	c_m = 2 (m - 1).
\]
Since
\[
	\prod_{k=1}^n a_k^{2n-2k-1} c_k^{k-1}
	= \prod_{k=1}^n 2^{2n-k-2} (k-1)^{k-1}
	= 2^{n(3n-5)/2} \prod_{k=1}^{n-1} k^k,
\]
we obtain the theorem by Theorem~\ref{thm:res-general}.
\end{proof}

\section{Compact formulas for discriminants of classical quasi-orthogonal polynomials} \label{sec:explicit}

In this section we derive explicit formulas for the discriminants of
quasi-Jacobi, quasi-Laguerre and quasi-Hermite polynomials.
The proof substantially uses the derivative properties of classical orthogonal polynomials.
We first derive a general result and then apply it to specific cases.

\begin{theorem}\label{thm:D-general}
Let $\Phi_n$ and $\Phi_{n-1}$ be polynomials of degree $n$ and $n-1$ respectively.
Assume that
\begin{equation}\label{eq:derivative}
\begin{split}
	\rho(x) \Phi'_n(x) &= (A_n x + B_n) \Phi_n(x) + C_n \Phi_{n-1}(x), \\
	\rho(x) \Phi'_{n-1}(x) &= (D_n x + E_n) \Phi_{n-1}(x) + F_n \Phi_n(x),
\end{split}
\end{equation}
where $\rho(x)$ is a polynomial and
$A_n$, $B_n$, $C_n$, $D_n$, $E_n$, $F_n$ are constants.
Let $c$ be a non-zero constant and let $\Phi_{n;c}(x)=\Phi_{n}(x)+c\Phi_{n-1}(x)$.
Let $l_n$ be the leading coefficient of $\Phi_n$.
Then
\begin{equation} \label{eq:D-general}
\begin{split}
	\disc(\Phi_{n;c})
	&= \frac{(-1)^{n(n+1)/2} (D_n-A_n)^n c^n}{l_n^{2-\deg\rho} \Res(\Phi_{n;c},\rho)}
	\Res(\Phi_n,\Phi_{n-1}) \Phi_{n;c}(\xi_{n;c}),
\end{split}
\end{equation}
where
\[
	\xi_{n;c} = \frac{F_n c^2+(B_n-E_n)c-C_n}{(D_n-A_n)c}.
\]
Furthermore, $\disc(\Phi_{n;c})$
is a polynomial in $c$ of degree at most $2(n-1)$.
\end{theorem}

\begin{proof}[Proof of Theorem~\ref{thm:D-general}]
By \eqref{eq:DR} and \eqref{eq:res-mult},
\begin{equation}\label{eq:DG1}
	\disc(\Phi_{n;c})
	= \frac{(-1)^{n(n-1)/2}}{l_n} \Res(\Phi_{n;c},\Phi'_{n;c})
	= \frac{(-1)^{n(n-1)/2}}{l_n}
	\frac{\Res(\Phi_{n;c},\rho\Phi'_{n;c})}{\Res(\Phi_{n;c},\rho)}.
\end{equation}
By \eqref{eq:derivative},
\begin{align*}
	\rho(x) \Phi'_{n;c}(x)
	&= (A_n x + F_n c + B_n) \Phi_n(x)
	+ (D_n c x + E_n c + C_n) \Phi_{n-1}(x) \\
	&= (A_n x + F_n c + B_n) \Phi_{n;c}(x)
	+ L(x) \Phi_{n-1}(x),
\end{align*}
where $L(x) = (D_n - A_n) c x - F_n c^2 - (B_n - E_n) c + C_n$.
Hence, by Lemma~\ref{lemma:Euclidean} and \eqref{eq:res-mult},
\begin{align}\label{eq:DG2}
	\Res(\Phi_{n;c}, \rho\Phi'_{n;c})
	&= l_n^{\deg\rho-1}\Res(\Phi_{n;c}, L \Phi_{n-1}) \nonumber \\
	&= l_n^{\deg\rho-1}\Res(\Phi_{n;c}, L) \Res(\Phi_{n;c}, \Phi_{n-1}) \nonumber \\
	&= l_n^{\deg\rho-1}\Res(\Phi_{n;c}, L) \Res(\Phi_n, \Phi_{n-1}).
\end{align}
Since $\xi_{n;c}$ is the root of $L$,
by \eqref{eq:resultant},
\begin{equation}\label{eq:DG3}
	\Res(\Phi_{n;c}, L)
	= (-1)^n (D_n - A_n)^n c^n
	\Phi_{n;c}\left(\xi_{n;c}\right).
\end{equation}
Therefore \eqref{eq:D-general} follows from \eqref{eq:DG1}, \eqref{eq:DG2} and \eqref{eq:DG3}.

The latter part of the theorem follows from Proposition~\ref{prop:DLC}.
\end{proof}

\begin{remark}
If $\{\Phi_n\}$ is a sequence of classical orthogonal polynomials,
then it satisfies \eqref{eq:derivative} for all $n$.
Conversely, let $\{\Phi_n\}$ be a sequence of polynomials
satisfying \eqref{eq:derivative} for all $n$.
Then we obtain the three-term relation \eqref{eq:rho2} by eliminating $\rho\Phi'_n$.
Al-Salam and Chihara~\cite{AC72} proved that
if $\{\Phi_n\}$ satisfies \eqref{eq:rho2} and \eqref{eq:derivative},
then $\Phi_n$ is a classical orthogonal polynomial or the Bessel polynomial.
\end{remark}

\subsection{Quasi-Jacobi polynomials} \label{subsec:QJ}

The discriminants of quasi-Jacobi polynomials are computed as follows.

\begin{theorem}\label{thm:DQJ}
Let $c$ be a constant and let
$P_{n;c}^{(\alpha,\beta)}(x)
=P_n^{(\alpha,\beta)}(x)+cP_{n-1}^{(\alpha,\beta)}(x)$.
Then
\begin{align} \label{eq:J-compact}
	\disc(P_{n;c}^{(\alpha,\beta)})
	&= \frac{(2n+\alpha+\beta)^{2n-1}}{2^{n(n-1)}}
	\prod_{k=1}^n k^{k-2n+3} \nonumber \\
	&\quad \cdot \prod_{k=1}^{n-1} (k+\alpha)^{k-1} (k+\beta)^{k-1}
	(n+k+\alpha+\beta)^{n-k-1} \\
	&\quad \cdot \frac{(-c)^n
	P_{n;c}^{(\alpha,\beta)}
	\left(-\frac{2 n (n + \alpha + \beta) c^2
	+ (\alpha^2 - \beta^2) c
	+ 2 (n + \alpha)(n + \beta)}{(2 n + \alpha + \beta)^2 c}\right)}
	{(n+\alpha+cn)(n+\beta-cn)}. \nonumber 
\end{align}
Furthermore, $\disc(P_{n;c}^{(\alpha,\beta)})$
is a polynomial in $c$ of degree $2(n-1)$.
\end{theorem}

\begin{remark} Taking the limit as $c\to 0$, we have
\[
	\disc(P_n^{(\alpha,\beta)}) = 2^{-n(n-1)} \prod_{k=1}^n k^{k-2n+2} (k+\alpha)^{k-1}(k+\beta)^{k-1} (n+k+\alpha+\beta)^{n-k}.
\]
This formula coincides with Stieltjes's formula~\cite[(6.71.5)]{Sze}.
\end{remark}

\begin{proof}[Proof of Theorem~\ref{thm:DQJ}]
By \eqref{eq:DEJ1a} and \eqref{eq:DEJ1b},
$P_n^{(\alpha,\beta)}$ and $P_{n-1}^{(\alpha,\beta)}$ satisfies
\eqref{eq:derivative} for $\rho(x)=1-x^2$ and suitable constants.
Furthermore, we have
\begin{equation}\label{eq:DQJ1}
\begin{split}
	D_n - A_n &= 2 n + \alpha + \beta, \\
	\xi_{n;c} &= -\frac{2 n (n + \alpha + \beta) c^2
	+ (\alpha^2 - \beta^2) c
	+ 2 (n + \alpha)(n + \beta)}{(2 n + \alpha + \beta)^2 c}.
\end{split}
\end{equation}

By \eqref{eq:resultant},
\[
	\Res(P_{n;c}^{(\alpha,\beta)}, \rho)
	= (-1)^n P_{n;c}^{(\alpha,\beta)}(1) P_{n;c}^{(\alpha,\beta)}(-1).
\]
By \eqref{eq:J-explicit},
\begin{align*}
	P_{n;c}^{(\alpha,\beta)}(1)
	&= \binom{n+\alpha}{n}+c\binom{n-1+\alpha}{n-1}, \\
	P_{n;c}^{(\alpha,\beta)}(-1)
	&= (-1)^n \left(\binom{n+\beta}{n}-c\binom{n-1+\beta}{n-1}\right).
\end{align*}
Hence we have
\begin{equation}\label{eq:DQJ2}
\begin{split}
	&\Res(P_{n;c}^{(\alpha,\beta)}, \rho) \\
	&= \left(\binom{n+\alpha}{n}+c\binom{n-1+\alpha}{n-1}\right)
	\left(\binom{n+\beta}{n}-c\binom{n-1+\beta}{n-1}\right) \\
	&= \frac{(n+\alpha+cn)(n+\beta-cn)}{(n!)^2}
	\prod_{k=1}^{n-1} (k+\alpha) (k+\beta).
\end{split}
\end{equation}

Therefore, by Theorems~\ref{thm:res-QJ} and \ref{thm:D-general}, \eqref{eq:DQJ1} and \eqref{eq:DQJ2},
\begin{align*}
	\disc(P_{n;c}^{(\alpha,\beta)})
	&= \frac{(-1)^{n(n+1)/2} (D_n-A_n)^n c^n}{\Res(P_{n;c},\rho)}
	\Res(P_n^{(\alpha,\beta)},P_{n-1}^{(\alpha,\beta)})
	P_{n;c}^{(\alpha,\beta)}(\xi_{n;c}) \\
	&= \frac{(2n+\alpha+\beta)^{2n-1} (-c)^n
	P_{n;c}^{(\alpha,\beta)}(\xi_{n;c}^{(\alpha,\beta)})}
	{2^{n(n-1)} (n+\alpha+cn)(n+\beta-cn)}
	\prod_{k=1}^n k^{k-2n+3} \\
	&\quad \cdot \prod_{k=1}^{n-1} (k+\alpha)^{k-1} (k+\beta)^{k-1}
	(n+k+\alpha+\beta)^{n-k-1}.
\end{align*}

The latter part of the corollary follows
from Propositions~\ref{prop:DLC} and~\ref{prop:zero3}.
\end{proof}

We now describe some specializations of Theorem~\ref{thm:DQJ}.

For $\lambda \in \rr$ and $0 < n \in \zz$, we define
\[
	(\lambda)_0 = 1, \quad
	(\lambda)_n = \lambda (\lambda+1) \dotsm (\lambda+n-1).
\]
The $n$-th \emph{Gegenbauer polynomial} is defined by
\begin{equation}
\label{eq:Gegenbauer}
	C^{(\lambda)}_n(x) =
	\frac{(2\lambda)_n}{(\lambda+\frac{1}{2})_n}
	P^{(\lambda-1/2,\lambda-1/2)}_n(x).
\end{equation}
These polynomials often appear in the study of spherical designs (cf.~\cite{BD79,DGS77}).

\begin{corollary}\label{cor:DQG}
Let $c$ be a constant and let
$C_{n;c}^{(\lambda)}(x)
=C^{(\lambda)}_n(x)+c C^{(\lambda)}_{n-1}(x)$.
Then
\begin{equation} \label{eq:G-compact}
\begin{gathered}
	\disc(C_{n;c}^{(\lambda)})
	= 2^{n(n-1)} (2n+2\lambda-1)^n
	\prod_{k=1}^n k^{k-2n+3} (k+\lambda-1)^{2n-2k} \\
	\cdot \prod_{k=1}^{n-1}
	(k+2\lambda-1)^{k-2}
	\cdot \frac{(-c)^n C^{(\lambda)}_{n;c}
	\left(-\frac{nc^2+n+2\lambda-1}{(2n+2\lambda-1)c}\right)}
	{(n+2\lambda-1)^2-(cn)^2}.
\end{gathered}
\end{equation}
Furthermore, $\disc(C_{n;c}^{(\lambda)})$ is an even polynomial in $c$ of degree $2(n-1)$.
\end{corollary}
\begin{proof}[Proof of Corollary~\ref{cor:DQG}]
By definition,
\begin{align*}
	C_{n;c}^{(\lambda)}(x)
	&= \frac{(2\lambda)_n}{(\lambda+\frac{1}{2})_n}
	P^{(\lambda-1/2,\lambda-1/2)}_n(x)
	+ c \frac{(2\lambda)_{n-1}}{(\lambda+\frac{1}{2})_{n-1}}
	P^{(\lambda-1/2,\lambda-1/2)}_{n-1}(x) \\
	&= \frac{(2\lambda)_n}{(\lambda+\frac{1}{2})_n}
	P^{(\lambda-1/2,\lambda-1/2)}_{n;c'}(x),
\end{align*}
where
\[
	c' = c \frac{\lambda+n-\frac{1}{2}}{2\lambda+n-1}.
\]
By Proposition~\ref{prop:DS4.3} and Theorem~\ref{thm:DQJ},
\begin{align*}
	\disc(C_{n;c}^{(\lambda)})
	&= \left(\frac{(2\lambda)_n}{(\lambda+\frac{1}{2})_n}\right)^{2(n-1)}
	\frac{(2n+2\lambda-1)^{2n-1}}{2^{n(n-1)}}
	\prod_{k=1}^n k^{k-2n+3} \\
	&\quad \cdot \prod_{k=1}^{n-1} \left(k+\lambda-\frac{1}{2}\right)^{2k-2}
	(n+k+2\lambda-1)^{n-k-1} \\
	&\quad \cdot \frac{(-c')^n
	P^{(\lambda-1/2,\lambda-1/2)}_{n;c'}
	\left(-\frac{2 n (n + 2\lambda - 1) (c')^2
	+ 2 (n + \lambda - 1/2)^2}{(2 n + 2\lambda - 1)^2 c'}\right)}
	{(n+\lambda-\frac{1}{2}+c'n)(n+\lambda-\frac{1}{2}-c'n)} \\
	&= \left(\frac{(2\lambda)_n}{(\lambda+\frac{1}{2})_n}\right)^{2n-3}
	\frac{(2n+2\lambda-1)^{2n-1}}{2^{n(n-1)}}
	\prod_{k=1}^n k^{k-2n+3} \\
	&\quad \cdot \prod_{k=1}^{n-1} \left(k+\lambda-\frac{1}{2}\right)^{2k-2}
	(n+k+2\lambda-1)^{n-k-1} \\
	&\quad \cdot \frac{(n+\lambda-\frac{1}{2})^{n-2}
	(-c)^n C^{(\lambda)}_{n;c}
	\left(-\frac{nc^2+n+2\lambda-1}{(2n+2\lambda-1)c}\right)}
	{(n+2\lambda-1)^{n-2} (n+2\lambda-1+cn)(n+2\lambda-1-cn)} \\
	&= \frac{(2n+2\lambda-1)^n}{2^{(n-1)^2} (n+2\lambda-1)^{n-2}}
	\prod_{k=1}^n k^{k-2n+3} (k+2\lambda-1)^{2n-3} \\
	&\quad \cdot \prod_{k=1}^{n-1}
	\left(k+\lambda-\frac{1}{2}\right)^{2k-2n+1}
	(n+k+2\lambda-1)^{n-k-1} \\
	&\quad \cdot \frac{(-c)^n C^{(\lambda)}_{n;c}
	\left(-\frac{nc^2+n+2\lambda-1}{(2n+2\lambda-1)c}\right)}
	{(n+2\lambda-1)^2-(cn)^2}.
\end{align*}
The constant factor is computed as follows:
\begin{align*}
	&\frac{(2n+2\lambda-1)^n}{2^{(n-1)^2} (n+2\lambda-1)^{n-2}}
	\prod_{k=1}^n k^{k-2n+3} (k+2\lambda-1)^{2n-3} \\
	&\quad \cdot \prod_{k=1}^{n-1}
	\left(k+\lambda-\frac{1}{2}\right)^{2k-2n+1}
	(n+k+2\lambda-1)^{n-k-1} \displaybreak[0]\\
	&= \frac{(2n+2\lambda-1)^n}{(n+2\lambda-1)^{n-2}}
	\prod_{k=1}^n k^{k-2n+3} (k+2\lambda-1)^{2n-3} \\
	&\quad \cdot \prod_{k=1}^{n-1}
	\left(2k+2\lambda-1\right)^{2k-2n+1}
	\cdot \prod_{k=n+1}^{2n-1} (k+2\lambda-1)^{2n-k-1} \displaybreak[0]\\
	&= \frac{(2n+2\lambda-1)^n}{(n+2\lambda-1)^{n-2}}
	\prod_{k=1}^n k^{k-2n+3} (k+2\lambda-1)^{2n-3} \\
	&\quad \cdot \prod_{k=1}^{2n-1}
	\left(k+2\lambda-1\right)^{k-2n+1}
	\cdot \prod_{k=1}^n (2k-1+2\lambda-1)^{2n-2k} \\
	&\quad \cdot \prod_{k=1}^{2n-1} (k+2\lambda-1)^{2n-k-1}
	\cdot \prod_{k=1}^n (k+2\lambda-1)^{k-2n+1} \displaybreak[0]\\
	&= 2^{n(n-1)} (2n+2\lambda-1)^n
	\prod_{k=1}^n k^{k-2n+3} (k+\lambda-1)^{2n-2k} \cdot
	\prod_{k=1}^{n-1} (k+2\lambda-1)^{k-2}.
\end{align*}

The latter part of the corollary follows
from Propositions~\ref{prop:DLC} and~\ref{prop:zero3}.
\end{proof}

We also describe another specialization of Theorem~\ref{thm:DQJ}.

\begin{corollary}\label{cor:DQC}
Let $c$ be a constant and let
$T_{n;c}(x)=T_n(x)+c T_{n-1}(x)$ and
$U_{n;c}(x)=U_n(x)+c U_{n-1}(x)$.
Then we have
\begin{equation}\label{eq:DC1}
	\disc(T_{n;c})
	= \frac{2^{(n-1)(n-2)} (2n-1)^n (-c)^n}{1-c^2}
	T_{n;c}\left(-\frac{(n-1)c^2+n}{(2n-1)c}\right).
\end{equation}
\begin{equation}\label{eq:DC2}
	\disc(U_{n;c})
	= \frac{2^{n(n-1)} (2n+1)^n (-c)^n }
	{(n+1)^2-(cn)^2}
	U_{n;c}\left(-\frac{nc^2+n+1}{(2n+1)c}\right).
\end{equation}
Furthermore, $\disc(T_{n;c})$ and $\disc(U_{n;c})$
are even polynomials in $c$ of degree $2(n-1)$.
\end{corollary}
\begin{proof}[Proof of Corollary~\ref{cor:DQC}]
We first consider $\disc(T_{n;c})$.
When $n=1$, it is easy to verify \eqref{eq:DC1}.
We assume that $n\ge 2$. By \eqref{eq:C1G},
\[
	T_{n;c}(x)
	= \lim_{\lambda\to 0} \frac{n}{2\lambda} C^{(\lambda)}_n(x)
	+ c \lim_{\lambda\to 0} \frac{n-1}{2\lambda} C^{(\lambda)}_{n-1}(x) 
        = \lim_{\lambda\to 0} \frac{n}{2\lambda} C^{(\lambda)}_{n;c'}(x),
\]
where $c'=(n-1)c/n$.
By Proposition~\ref{prop:DS4.3} and Corollary~\ref{cor:DQG},
\begin{align*}
	\disc(T_{n;c}^{(\lambda)})
	&= \lim_{\lambda\to 0} \left(\frac{n}{2\lambda}\right)^{2n-2}
	2^{n(n-1)} (2n+2\lambda-1)^n
	\prod_{k=1}^n k^{k-2n+3} (k+\lambda-1)^{2n-2k} \\
	&\quad \cdot \prod_{k=1}^{n-1}
	(k+2\lambda-1)^{k-2}
	\cdot \frac{(-c')^n C^{(\lambda)}_{n;c'}
	\left(-\frac{n(c')^2+n+2\lambda-1}{(2n+2\lambda-1)c'}\right)}
	{(n+2\lambda-1)^2-(c'n)^2} \displaybreak[0]\\
	&= \lim_{\lambda\to 0}
	2^{(n-1)(n-2)} (n-1)^n n^{n-3} (2n+2\lambda-1)^n \\
	&\quad \cdot \prod_{k=2}^n k^{k-2n+3} (k+\lambda-1)^{2n-2k} \\
	&\quad \cdot \prod_{k=2}^{n-1}
	(k+2\lambda-1)^{k-2}
	\cdot \frac{(-c)^n \frac{n}{2\lambda}C^{(\lambda)}_{n;c'}
	\left(-\frac{n(c')^2+n+2\lambda-1}{(2n+2\lambda-1)c'}\right)}
	{(n+2\lambda-1)^2-c^2(n-1)^2} \displaybreak[0]\\
	&= 2^{(n-1)(n-2)} (n-1)^n n^{n-3} (2n-1)^n
	\prod_{k=2}^n k^{k-2n+3} (k-1)^{2n-2k} \\
	&\quad \cdot \prod_{k=2}^{n-1}
	(k-1)^{k-2}
	\cdot \frac{(-c)^n T_{n;c}
	\left(-\frac{(n-1)c^2+n}{(2n-1)c}\right)}
	{(n-1)^2-c^2(n-1)^2} \displaybreak[0]\\
	&= \frac{2^{(n-1)(n-2)} (2n-1)^n (-c)^n}{1-c^2}
	T_{n;c}\left(-\frac{(n-1)c^2+n}{(2n-1)c}\right).
\end{align*}

Next, we consider $\disc(U_{n;c})$.
Since $U_{n;c}(x)=C^{(1)}_{n;c}(x)$,
by Corollary~\ref{cor:DQG},
\begin{align*}
	\disc(U_{n;c})
	&= 2^{n(n-1)} (2n+1)^n
	\prod_{k=1}^n k^{3-k} \\
	&\quad \cdot \prod_{k=1}^{n-1}
	(k+1)^{k-2}
	\cdot \frac{(-c)^n U_{n;c}
	\left(-\frac{nc^2+n+1}{(2n+1)c}\right)}
	{(n+1)^2-(cn)^2}. \\
	&= \frac{2^{n(n-1)} (2n+1)^n (-c)^n }
	{(n+1)^2-(cn)^2}
	U_{n;c}\left(-\frac{nc^2+n+1}{(2n+1)c}\right). \qedhere
\end{align*}
\end{proof}

\begin{remark}
Dilcher and Stolarsky~\cite[Theorem~4]{DS05} derived the compact formula \eqref{eq:DC2}
by using algebraic properties of resultants and the Euclidean algorithm.
\end{remark}

\subsection{Quasi-Laguerre and quasi-Hermite polynomials} \label{subsec:QL-QH}

In this subsection we derive explicit formulas for the discriminants of quasi-Laguerre and quasi-Hermite polynomials, respectively.

\begin{theorem}\label{thm:DQL}
Let $c$ be a constant and let
$L_{n;c}^{(\alpha)}(x)=L_n^{(\alpha)}(x)+cL_{n-1}^{(\alpha)}(x)$.
Then
\begin{equation} \label{eq:L-compact}
\begin{gathered}
	\disc(L_{n;c}^{(\alpha)})
	= \frac{1}{n+\alpha+cn} \prod_{k=1}^n k^{k-2n+3}
	\prod_{k=1}^{n-1} (k+\alpha)^{k-1} \\
	\quad \cdot (-c)^n L_{n;c}^{(\alpha)}
	\left(\frac{n c^2+(2n+\alpha)c+n+\alpha}{c}\right).
\end{gathered}
\end{equation}
Furthermore, $\disc(L_{n;c}^{(\alpha)})$ is a polynomial in $c$ of degree $2(n-1)$.
\end{theorem}

\begin{proof}[Proof of Theorem~\ref{thm:DQL}]
By \eqref{eq:L-rec} and \eqref{eq:DEL1},
$L_n^{(\alpha)}$ and $L_{n-1}^{(\alpha)}$ satisfies
\eqref{eq:derivative} for $\rho(x)=x$ and suitable constants.
In fact, we have
\begin{align*}
	x \frac{d}{dx} L_n^{(\alpha)}(x)
	&= n L_n^{(\alpha)}(x) - (n + \alpha) L_{n-1}^{(\alpha)}(x), \\
	x \frac{d}{dx} L_{n-1}^{(\alpha)}(x)
	&= (x - n - \alpha) L_{n-1}^{(\alpha)}(x) + n L_n^{(\alpha)}(x).
\end{align*}
Hence we have
\begin{equation}\label{eq:DQL1}
	D_n - A_n = 1, \quad
	\xi_{n;c} = \frac{n c^2 + (2 n + \alpha) c + n + \alpha}{c}.
\end{equation}

By \eqref{eq:resultant} and \eqref{eq:L-explicit},
\begin{multline}\label{eq:DQL2}
	\Res(L_{n;c}^{(\alpha)}, \rho) = (-1)^n L_{n;c}^{(\alpha)}(0) \\
	= (-1)^n \left(\binom{n+\alpha}{n}
	+ c \binom{n+\alpha-1}{n-1}\right)
	= \frac{(-1)^n(n+\alpha+cn)}{n!}\prod_{k=1}^{n-1}(k+\alpha).
\end{multline}
Let $l_n$ be the leading coefficient of $L_n^{(\alpha)}$.
By Theorem~\ref{thm:res-QL},
\begin{equation}\label{eq:DQL3}
	\Res(L^{(\alpha)}_n, L^{(\alpha)}_{n-1})
	= (-1)^{n(n+1)/2} l_n
	\prod_{k=1}^{n} k^{k-2n+2}
	\prod_{k=1}^{n-1} (k+\alpha)^k.
\end{equation}

Therefore, by Theorem~\ref{thm:D-general},
\eqref{eq:DQL1}, \eqref{eq:DQL2} and \eqref{eq:DQL3},
\begin{align*}
	\disc(L_n^{(\alpha)})
	&= \frac{(-1)^{n(n+1)/2} (D_n-A_n)^n c^n}
	{l_n\Res(L_{n;c}^{(\alpha)},\rho)}
	\Res(L_n^{(\alpha)},L_{n-1}^{(\alpha)})
	L_n^{(\alpha)}(\xi_{n;c}) \\
	&= \frac{(-c)^n}{n+\alpha+cn}
	\prod_{k=1}^n k^{k-2n+3} \prod_{k=1}^{n-1} (k+\alpha)^{k-1}
	\cdot L_n^{(\alpha)}(\xi_{n;c}). \qedhere
\end{align*}
\end{proof}

\begin{remark}
Taking the limit as $c\to 0$, we have
\[
	\disc(L_n^{(\alpha)}) = \prod_{k=1}^n k^{k-2n+2} (k+\alpha)^{k-1}.
\]
This formula coincides with Stieltjes's formula~\cite[(6.71.6)]{Sze}.
\end{remark}

Next, we derive an explicit formula for the discriminants of quasi-Hermite polynomials.

\begin{theorem}\label{thm:DQH}
Let $c$ be a constant and let $H_{n;c}(x)=H_n(x)+cH_{n-1}(x)$.
Then
\begin{equation}
\label{eq:H-compact}
	\disc(H_{n;c}) = 2^{n(3n-5)/2} \prod_{k=1}^{n-1} k^k
	\cdot (-c)^n H_{n;c}\left(-\frac{c^2+2n}{2c}\right).
\end{equation}
Furthermore, $\disc(H_{n;c})$ is an even polynomial in $c$ of degree $2(n-1)$.
\end{theorem}

We give a proof by using the limiting property (cf.~\cite[(5.6.3)]{Sze}) that
\begin{equation}
\label{eq:HG}
	\frac{H_n(x)}{n!} = \lim_{\lambda\to\infty} \lambda^{-n/2}
	C_n^{(\lambda)}(\lambda^{-1/2} x).
\end{equation}

\begin{proof}[Proof of Theorem~\ref{thm:DQH}]
By \eqref{eq:H-rec} and \eqref{eq:DEH1},
$H_n$ and $H_{n-1}$ satisfies
\eqref{eq:derivative} for $\rho(x)=1$ and suitable constants.
In fact, we have
\begin{align*}
	H'_n(x) &= 2 n H_{n-1}(x), \\
	H'_{n-1}(x) &= 2 x H_{n-1}(x) - H_n(x).
\end{align*}
Hence we have
\begin{equation}\label{eq:DQH1}
	D_n - A_n = 2, \quad
	\xi_{n;c} = -\frac{c^2+2n}{2c}.
\end{equation}

Let $l_n$ be the leading coefficient of $H_n$.
By \eqref{eq:H-explicit} and \eqref{eq:resultant},
\begin{equation}\label{eq:DQH2}
	l_n = 2^n, \quad
	\Res(H_{n;c}, \rho) = 1.
\end{equation}
By Theorem~\ref{thm:res-QH},
\begin{equation}\label{eq:DQH3}
	\Res(H_n, H_{n-1})
	= (-1)^{n(n-1)/2} 2^{n(3n-5)/2} l_n
	\prod_{k=1}^{n-1} k^k.
\end{equation}

Therefore, by Theorem~\ref{thm:D-general},
\eqref{eq:DQH1}, \eqref{eq:DQH2} and \eqref{eq:DQH3},
\begin{align*}
	\disc(H_n)
	&= \frac{(-1)^{n(n+1)/2} (D_n-A_n)^n c^n}
	{l_n^2\Res(H_{n;c},\rho)}
	\Res(H_n,H_{n-1})
	H_n(\xi_{n;c}) \\
	&= 2^{n(3n-5)/2} \prod_{k=1}^{n-1} k^k
	\cdot (-c)^n H_n(\xi_{n;c}).
\end{align*}

The latter part of the theorem follows
from Propositions~\ref{prop:DLC} and~\ref{prop:zero3}.
\end{proof}

\begin{remark}
\label{rem:Hilbert}
Taking the limit as $c\to 0$, we have
\[
	\disc(H_n) = 2^{3n(n-1)/2} \prod_{k=1}^n k^k.
\]
This formula coincides with Hilbert's formula~\cite[(6.71.7)]{Sze}.
\end{remark}

\begin{remark}
Laguerre polynomials are expressed as a limit case of
Jacobi polynomials (see \cite[(5.3.4)]{Sze}):
\[
	L_n^{(\alpha)}(x) = \lim_{\beta\to\infty}
	P_n^{(\alpha,\beta)}(1 - 2 \beta^{-1} x).
\]
Similarly, Hermite polynomials are expressed as a limit case of
Gegenbauer polynomials or Laguerre polynomials
(see \cite[(5.6.3) and p.~389, Problem~80]{Sze}):
\begin{gather*}
	\frac{H_n(x)}{n!} = \lim_{\lambda\to\infty} \lambda^{-n/2}
	C_n^{(\lambda)}(\lambda^{-1/2} x), \\
	\lim_{\alpha\to\infty} \alpha^{-n/2}
	L_n^{(\alpha)}(\alpha^{1/2} x + \alpha)
	= (-1)^n 2^{-n/2} (n!)^{-1} H_n(2^{-1/2} x).
\end{gather*}
By these relations, together with Proposition~\ref{prop:DS4.3},
we can give alternative proofs of Theorems~{\rm \ref{thm:DQL}} and \ref{thm:DQH}.
\end{remark}

\section{Number-theoretic applications} \label{sec:applications}

In this section we give a generalization of Hausdorff's equations (\ref{eq:hausdorff-equation1})
and then, in some specific cases, examine solutions for such equations.
We use the explicit formula for discriminants of quasi-Hermite polynomials given in Theorem~\ref{thm:DQH}.

Throughout this section, let
\begin{equation}
\label{eq:gauss-moment-1}
a_k = \frac{1}{\sqrt{\pi}} \int_{-\infty}^\infty t^k \ e^{-t^2}dt, \qquad k=0,1,\ldots
\end{equation}
It is then obvious that
\begin{equation}
\label{eq:gauss-moment-2}
a_{2k} = \frac{(2k)!}{2^{2k} k!}, \qquad a_{2k+1} = 0.
\end{equation}

\subsection{Hausdorff-type equations} \label{subsec:HDeq}

The following is a generalization of the equations \eqref{eq:hausdorff-equation1}:

\begin{problem}
[Hausdorff-type equations]
\label{prob:haus-eq-1}
Let $m>0$ and $n \ge 0$ be integers. Do the Diophantine equations
\begin{equation}
\label{eq:haus-eq-1}
\left\{\begin{array}{ccccccccc}
x_1       & + & x_2       & + & \cdots & + & x_m        & = & a_0 \\
x_1y_1    & + & x_2y_2    & + & \cdots & + & x_m y_m    & = & a_1 \\
&&&&&&&& \vdots \\
x_1y_1^n  & + & x_2y_2^n  & + & \cdots & + & x_m y_m^n  & = & a_n \\
\end{array} \right.
\end{equation}
have a solution $(x_1, \ldots, x_m, y_1, \ldots, y_m) \in \qq^{2m}$?
\end{problem}

The following proposition makes the relationship of Problem~\ref{prob:haus-eq-1} to
quadrature formulas for Gaussian integration:

\begin{proposition}
\label{prop:haus-equiv-1}
The following are equivalent:
\begin{enumerate}
\item[{\rm (i)}] The equations {\rm (\ref{eq:haus-eq-1})} have a solution $(x_1,\ldots,x_m,y_1,\ldots,y_m) \in \qq^{2m}$;
\item[{\rm (ii)}] The formula
\begin{equation}
\label{eq:haus-qf-1}
   \frac{1}{\sqrt{\pi}} \int_{-\infty}^\infty f(t) \ e^{-t^2}dt = \sum_{i=1}^m x_i f(y_i)
\end{equation}
is a rational quadrature of degree $n$.
\end{enumerate}
\end{proposition}
\begin{proof}[Proof of Proposition~\ref{prop:haus-equiv-1}]
We remark that $1,x,x^2,\ldots,x^n$ form a basis of the vector space of all polynomials of degree at most $n$.
\end{proof}

The following proposition gives a slight generalization of
the {\it Stroud bound} for positive quadrature formulas~\cite{S71} (see also~\cite[p.465]{Sho37}) or
\emph{Fisher-type inequality} for Gaussian designs~\cite{BB04}:

\begin{proposition}
\label{prop:haus-bound-1}
If there exists a rational solution of (\ref{eq:haus-eq-1}), then $n \le 2m-1$.
\end{proposition}
\begin{proof}[Proof of Proposition~\ref{prop:haus-bound-1}]
Suppose contrary.
Let $f$ be a polynomial which vanishes at all $y_i$.
Then
\[
0 < \frac{1}{\sqrt{\pi}} \int_{-\infty}^\infty f^2 \ e^{-t^2}dt = \sum_{i=1}^m x_i (f(y_i))^2 = 0,
\]
which is clearly a contradiction.
\end{proof}

The first pair $(m,n)$ to consider is that $n = 2m-1$.
Formulas of type (\ref{eq:haus-qf-1}) are then called {\it Gaussian quadrature} and
the nodes $y_i$ are the zeros of the Hermite polynomial $H_m$ (cf.~\cite{Sze}).
By a classical result of Schur~\cite{Sch29} (see also~\cite{S95}),
the polynomials $H_{2r}(x)$ and $H_{2r+1}(x)/x$ are irreducible over $\mathbb{Q}$.
So in this case, the equations (\ref{eq:haus-eq-1}) have no rational solutions.

The next case to consider is the
^^ almost tight' situation, namely the case when $n = 2m-2$.

The following proposition creates a relationship between the zeros of a quasi-Hermite polynomial and Eq.(\ref{eq:haus-eq-1}) for $n=2m-2$:

\begin{proposition}
\label{prop:ratio-shohat}
Assume that $n = 2m-2$.
Let $y_1,\ldots,y_m$ be distinct rationals.
The following are equivalent:
\begin{enumerate}
\item[{\rm (i)}] The equations {\rm (\ref{eq:haus-eq-1})} have
a solution $(x_1,\ldots,x_m,y_1,\ldots,y_m) \in \qq^{2m}$;
\item[{\rm (ii)}] 
There exists $c \in \qq$ such that $y_1,\ldots,y_m$ are the zeros of
the quasi-Hermite polynomial $H_{m;c}(x) = H_m(x) + cH_{m-1}(x)$.
\end{enumerate}
\end{proposition}
\begin{proof}[Proof of Proposition~\ref{prop:ratio-shohat}]
By (\ref{eq:H-explicit}), we remark that $H_m(x)$ is a polynomial with rational coefficients.
The result then follows by Theorem~\ref{thm:Shohat} and Proposition~\ref{prop:haus-equiv-1}.
\end{proof}

In the next subsection we prove a nonexistence theorem of solutions for $n = 2m-2$.
For this purpose, we substantially prove the nonexistence of
rational points on a certain hyperelliptic curve associated
with the discriminant ${\rm disc}(H_{m;c})$.

We work with the $2$-adic numbers $\qq_2$ rather than the rationals $\qq$.
Let $v_2\colon \qq_2^\times \to \zz$ be the normalized valuation,
where $\qq_2^\times$ is the set of units in $\qq_2$.
We use the convention that $v_2(0)=\infty$.
We denote by $\zz_2$ and $\zz_2^\times$ the ring of $2$-adic integers
and the set of units in $\zz_2$, respectively.
We remark that
\[
	\zz_2 = \{ x \in \qq_2 \mid v_2(x) \ge 0\}, \quad
	\zz_2^\times = \{ x \in \qq_2 \mid v_2(x) = 0\}.
\]

The following is used to show the main theorem in Subsection~\ref{subsec:nonexistence}
(cf.~\cite[Chapter~II, Theorem~4]{Serre}):

\begin{lemma}
\label{lem:square}
Let $x=2^n u$ be an element in $\mathbb{Q}_2^\times$ with
$n\in\mathbb{Z}$ and $u\in\mathbb{Z}_2^\times$.
For $x$ to be a square in $\mathbb{Q}_2$ it is necessary and sufficient that
$n$ is even and $u\equiv 1\pmod{8}$.
\end{lemma}

\subsection{Nonexistence theorem} \label{subsec:nonexistence}

The following is the main theorem in this subsection:

\begin{theorem} \label{thm:not_square}
If $n\equiv 3,4,5,6,7\pmod{8}$, then $\disc(H_{n;c})$ is not a square in
$\mathbb{Q}_2$ for any $c\in\mathbb{Q}_2$.
\end{theorem}

As a consequence, we obtain the following result:

\begin{corollary} \label{cor:non-exist}
If $r\equiv 2,3,4,5,6\pmod{8}$, then there do not exist rationals
$x_1$, \ldots, $x_{r+1}$, $y_1$, \ldots, $y_{r+1}$ such that
\begin{equation}\label{eq:D2r}
	\sum_{i=1}^{r+1} x_i y_i^k = a_k,
	\quad k=0,1,\dotsc,2r.
\end{equation}
\end{corollary}
\begin{proof}[Proof of Corollary~\ref{cor:non-exist}]
Assume that $x_1$, \ldots, $x_{r+1}$,
$y_1$, \ldots, $y_{r+1}$ are a rational solution of \eqref{eq:D2r}.
Then by Proposition~\ref{prop:haus-bound-1}, 
$y_1$, \ldots, $y_{r+1}$ are distinct from each other.
By Proposition~\ref{prop:ratio-shohat}
there exists $c \in \qq$ such that
the zeros of $H_{r+1;c}(x)$ are $y_1$, \ldots, $y_{r+1}$.
Therefore $\disc(H_{r+1;c})$ is a square in the rationals by (\ref{eq:def-disc}),
which however contradicts Theorem~\ref{thm:not_square}.
\end{proof}

\begin{proof}[Proof of Theorem~\ref{thm:not_square}]
Let
\[
	D_n(c) = (-c)^n H_{n;c}\left(-\frac{c^2+2n}{2c}\right).
\]
By Theorem~\ref{thm:DQH},
\begin{equation}\label{eq:DQH}
	\disc(H_{n;c}) = 2^{n(3n-5)/2} \prod_{k=1}^{n-1} k^k \cdot D_n(c).
\end{equation}
It is easily seen that
\begin{equation}\label{eq:exponent}
	\frac{n (3 n - 5)}{2} \equiv \begin{cases}
		0 \pmod{2} & \text{if $n\equiv 0,3 \pmod{4}$,} \\
		1 \pmod{2} & \text{if $n\equiv 1,2 \pmod{4}$.}
	\end{cases}
\end{equation}
Let $t=2\cdot 3\cdot 4^2\cdot 5^2\cdot\dotsm (n-1)^{\lfloor(n-1)/2\rfloor}$.
Then we have
\[
	\prod_{k=1}^{n-1} k^k =	 \begin{cases}
		t^2 \cdot (n - 1)!! & \text{if $n$ is even,} \\
		t^2 \cdot (n - 2)!! & \text{if $n$ is odd.}
	\end{cases}
\]
By Lemma~\ref{lem:square}, we have $2^{-2v_2(t)}t^2\equiv 1\pmod{8}$.
Since $1\cdot 3\cdot 5\cdot 7\equiv 1\pmod{8}$,
\begin{equation}\label{eq:prodk}
	2^{-2v_2(t)} \prod_{k=1}^{n-1} k^k
	\equiv \begin{cases}
		1 \pmod{8} & \text{if $n\equiv 0,1,2,3\pmod{8}$,} \\
		3 \pmod{8} & \text{if $n\equiv 4,5\pmod{8}$,} \\
		7 \pmod{8} & \text{if $n\equiv 6,7\pmod{8}$.}
	\end{cases}
\end{equation}

By \eqref{eq:H-explicit},
\begin{equation}\label{eq:Dnc}
	D_n(c) = \sum_{k=0}^{\lfloor n/2\rfloor} a_k
	- \sum_{k=0}^{\lfloor(n-1)/2\rfloor} b_k,
\end{equation}
where
\begin{equation}
 \label{eq:ak}
\begin{gathered}
	a_k = (-1)^k \frac{n!}{k!(n-2k)!} c^{2k} (c^2 + 2 n)^{n-2k}, \\
	b_k = (-1)^k \frac{(n-1)!}{k!(n-1-2k)!} c^{2k+2} (c^2 + 2 n)^{n-1-2k}.
\end{gathered}
\end{equation}
\begin{equation}\label{eq:binomial}
	\frac{n!}{k!(n-2k)!} = \frac{n!}{(2k)!(n-2k)!}\cdot\frac{(2k)!}{k!}
	= \binom{n}{2k} 2^k (2k-1)!!,
\end{equation}
we have
\begin{equation}
\label{eq:v2ak}
\begin{gathered}
	v_2(a_k) = v_2\left(\binom{n}{2k}\right) + m_n(c) k + n v_2(c^2 + 2 n), \\
	v_2(b_k) = v_2\left(\binom{n-1}{2k}\right) 	+ m_n(c) k + 2 v_2(c) + (n - 1) v_2(c^2 + 2 n),
\end{gathered}
\end{equation}
where $m_n(c) = 1 + 2 v_2(c) - 2 v_2(c^2 + 2 n)$.

By \eqref{eq:ak},
\begin{equation}
\label{eq:ab0}
	a_0 - b_0 = 2 n (c^2 + 2 n)^{n-1}, \quad
	a_1 - b_1 = -2 (n - 1) c^2 (c^2 + n^2) (c^2 + 2 n)^{n-3}.
\end{equation}
By expanding the right-hand sides,
\[
	a_0 - b_0 = 2 n c^{2n-2} + 4 p_0(c), \quad a_1 - b_1 = -2 (n - 1) c^{2n-2} + 4 p_1(c),
\]
where $p_0(c)$ and $p_1(c)$ are polynomials in $c$ of degree $2n-4$ with integer coefficients.
By \eqref{eq:binomial}, we have $a_k,b_k\in 4\mathbb{Z}[c]$ for $k\ge 2$.
The degrees of $a_k$ and $b_k$ in $c$ are $2n-2k$ by definition.
Therefore, by \eqref{eq:Dnc},
\begin{equation}\label{eq:expansion}
	D_n(c) = 2 c^{2n-2} + 4 s_{n-2} c^{2n-4} + \dotsb + 4 s_1 c^2 + 4 s_0,
\end{equation}
where $s_i \in \zz$; if $v_2(c)\le 0$, then $v_2(D_n(c))=v_2(2c^{2n-2})=2(n-1)v_2(c)+1$.

We divide the proof into four cases.


\bigskip

\noindent
\textbf{The case $n\equiv 3,7\pmod{8}$.} 
If $v_2(c)\le 0$, then $v_2(D_n(c))=2(n-1)v_2(c)+1$
by \eqref{eq:expansion}.
By \eqref{eq:DQH}, \eqref{eq:exponent} and \eqref{eq:prodk}, we have $v_2(\disc(H_{n;c}))\equiv 1\pmod{2}$.
Therefore $\disc(H_{n;c})$ is not a square in $\mathbb{Q}_2$ by Lemma~\ref{lem:square}.

If $v_2(c)\ge 1$, then $v_2(c^2+2n)=1$ and $m_n(c)=2v_2(c)-1\ge 1$.
By \eqref{eq:v2ak},
we have $v_2(a_k)\ge v_2(a_0)+1$ and $v_2(b_k)\ge v_2(b_0)+1$ for $k\ge 1$.
Since $v_2(b_0)-v_2(a_0)=2v_2(c)-1\ge 1$,
we have $v_2(D_n(c))=v_2(a_0)=n$.
Since $n\equiv 3,7\pmod{8}$,
we have $v_2(\disc(H_{n;c}))\equiv 1\pmod{2}$
by \eqref{eq:DQH}, \eqref{eq:exponent} and \eqref{eq:prodk}.
Therefore $\disc(H_{n;c})$ is not a square in $\mathbb{Q}_2$
by Lemma~\ref{lem:square}.


\bigskip

\noindent
\textbf{The case $n\equiv 5\pmod{8}$.} If $v_2(c)\le -1$, then by \eqref{eq:expansion}
\[
	\frac{D_n(c)}{2 c^{2n-2}} = 1 + 2 s_{n-2} c^{-2} + \dotsb
	+ 2 s_1 c^{-2n} + 2 s_0 c^{-2n+2}
	\equiv 1 \pmod{8}.
\]
By Lemma~\ref{lem:square}, we get
$c^{2n-2}/2^{2(n-1)v_2(c)}\equiv 1\pmod{8}$ and so $D_n(c)/2^{2(n-1)v_2(c)+1}\equiv 1\pmod{8}$.
By \eqref{eq:DQH}, \eqref{eq:exponent} and \eqref{eq:prodk}, we have
\[
	2^{2e} \disc(H_{n;c}) \equiv 3 \cdot 1 \equiv 3 \pmod{8},
\]
where $e$ is an integer.
Therefore $\disc(H_{n;c})$ is not a square in $\mathbb{Q}_2$
by Lemma~\ref{lem:square}.

If $v_2(c)=0$, then $v_2(c^2+2n)=0$ and $m_n(c)=1$.
Since $n-1$ is even, we have $(c^2+2n)^{n-1}\equiv 1\pmod{8}$
by Lemma~\ref{lem:square}.
By \eqref{eq:ab0},
\[
	\frac{a_0 - b_0}{2} = n (c^2 + 2 n)^{n-1}
	\equiv 5 \cdot 1 \equiv 5 \pmod{8}.
\]
Since $n\equiv 5\pmod{8}$, and
by \eqref{eq:v2ak},
we have
\begin{align*}
	& v_2(a_1) = v_2(a_2) = 2,\quad v_2(a_k) = v_2\left(\binom{n}{2k}\right) + k \ge 3, \\
       & v_2(b_1) = v_2(b_2) = 2,\quad v_2(b_k) = v_2\left(\binom{n-1}{2k}\right) + k \ge 3
\end{align*}
for $k \ge 3$.
Therefore, by \eqref{eq:Dnc},
\[
	D_n(c) \equiv a_0 - b_0 + a_1 - b_1 + a_2 - b_2
	\equiv 2 \cdot 5 + 4 - 4 + 4 - 4 \equiv 2 \pmod{8},
\]
and so $D_n(c)/2\equiv 1\pmod{4}$.
By \eqref{eq:DQH}, \eqref{eq:exponent} and \eqref{eq:prodk}, we have
\[
	2^{2e} \disc(H_{n;c}) \equiv 3 \cdot 1 \equiv 3 \pmod{4},
\]
where $e$ is an integer.
Therefore $\disc(H_{n;c})$ is not a square in $\mathbb{Q}_2$
by Lemma~\ref{lem:square}.

If $v_2(c)\ge 1$, then $v_2(c^2+2n)=1$ and $m_n(c)=2v_2(c)-1\ge 1$.
Since $n\equiv 5\pmod{8}$,
by \eqref{eq:v2ak},
we have
\begin{align*}
	v_2(a_k) &= v_2\left(\binom{n}{2k}\right) + m_n(c) k + n \ge n + 2, \\
	v_2(b_k) &= v_2\left(\binom{n-1}{2k}\right) + m_n(c) k + 2 v_2(c) + n - 1
	\ge n + 3
\end{align*}
for $k\ge 1$.
Hence we have $a_k/2^n\equiv b_k/2^n\equiv 0\pmod{4}$ for $k\ge 1$.
Since $n-1$ is even, $(c^2+2n)^{n-1}/2^{n-1}\equiv 1\pmod{8}$
by Lemma~\ref{lem:square}.
By \eqref{eq:ab0},
\[
	\frac{a_0 - b_0}{2^n} = n \cdot \frac{(c^2 + 2 n)^{n-1}}{2^{n-1}}
	\equiv 5 \pmod{8}.
\]
Therefore, by \eqref{eq:Dnc},
\[
	\frac{D_n(c)}{2^n} \equiv \frac{a_0 - b_0}{2^n}
	\equiv 1 \pmod{4}.
\]
By \eqref{eq:DQH}, \eqref{eq:exponent} and \eqref{eq:prodk}, we have
\[
	2^{2e} \disc(H_{n;c}) \equiv 3 \cdot 1 \equiv 3 \pmod{4},
\]
where $e$ is an integer.
Therefore $\disc(H_{n;c})$ is not a square in $\mathbb{Q}_2$
by Lemma~\ref{lem:square}.


\bigskip

\noindent
\textbf{The case $n\equiv 4\pmod{8}$.}
If $v_2(c)\le 0$, then $v_2(D_n(c))=2(n-1)v_2(c)+1$ by \eqref{eq:expansion}.
By  \eqref{eq:DQH}, \eqref{eq:exponent} and \eqref{eq:prodk}, we have $v_2(\disc(H_{n;c}))\equiv 1\pmod{2}$.
Therefore $\disc(H_{n;c})$ is not a square in $\mathbb{Q}_2$
by Lemma~\ref{lem:square}.

If $v_2(c)=1$, then $v_2(c^2+2n)=2$ and $m_n(c)=-1$.
Since $n\equiv 4\pmod{8}$,
and by \eqref{eq:v2ak},
\begin{gather*}
	v_2(a_{n/2}) = \frac{3}{2} n, \quad
	v_2(a_{n/2-1}) = \frac{3}{2} n + 2, \quad
	v_2(a_k) \ge \frac{3}{2} n + 2, \\
	v_2(b_{n/2-1}) = \frac{3}{2} n + 1, \quad
	v_2(b_k) \ge \frac{3}{2} n + 2
\end{gather*}
for $k\le n/2-2$.

By \eqref{eq:ak} and \eqref{eq:binomial},
\begin{equation}\label{eq:an2}
	a_{n/2} = (-1)^{n/2} \frac{n!}{(n/2)! 0!} c^n
	= 2^{n/2} (n-1)!! c^n
\end{equation}
Since $n\equiv 4\pmod{8}$, we have $(n-1)!!\equiv 3\pmod{8}$
and $c^n/2^n\equiv 1\pmod{8}$. Hence
$a_{n/2}/2^{3n/2}\equiv 3\pmod{8}$.
Therefore, by \eqref{eq:Dnc},
\[
	\frac{D_n(c)}{2^{3n/2}} \equiv
	\frac{a_{n/2} - b_{n/2-1}}{2^{3n/2}}
	\equiv 3 - 2
	\equiv 1 \pmod{4}.
\]
By \eqref{eq:DQH}, \eqref{eq:exponent} and \eqref{eq:prodk},
\[
	2^{2e} \disc(H_{n;c}) \equiv 3 \cdot 1 \equiv 3 \pmod{4},
\]
where $e$ is an integer.
Therefore $\disc(H_{n;c})$ is not a square in $\mathbb{Q}_2$
by Lemma~\ref{lem:square}.

If $v_2(c)=2$, then $v_2(c^2+2n)=3$ and $m_n(c)=-1$.
By \eqref{eq:v2ak},
\begin{gather*}
	v_2(a_{n/2}) = \frac{5}{2} n, \quad
	v_2(a_{n/2-1}) = v_2(a_{n/2-2}) = \frac{5}{2} n + 2,
	v_2(a_k) \ge \frac{5}{2} n + 3, \\
	v_2(b_{n/2-1}) = \frac{5}{2} n + 2, \quad
	v_2(b_{n/2-2}) = \frac{5}{2} n + 3, \quad
	v_2(b_k) \ge \frac{5}{2} n + 4
\end{gather*}
for $k\le n/2-3$.
Since $v_2(c)=2$ and $n$ is even, $c^n/2^{2n}\equiv 1\pmod{8}$.
By \eqref{eq:an2}, we have $a_{n/2}/2^{5n/2}\equiv 3\pmod{8}$.
Hence, by \eqref{eq:Dnc} and \eqref{eq:an2},
\[
	\frac{D_n(c)}{2^{5n/2}} \equiv
	\frac{a_{n/2} + a_{n/2-1} + a_{n/2-2} - b_{n/2-1}}{2^{5n/2}} \equiv
       3 + 4 + 4 - 4 \equiv
       7 \pmod{8}.
\]
By \eqref{eq:DQH}, \eqref{eq:exponent} and \eqref{eq:prodk}, we have
\[
	2^{2e} \disc(H_{n;c}) \equiv 3 \cdot 7 \equiv 5 \pmod{8},
\]
where $e$ is an integer.
Therefore $\disc(H_{n;c})$ is not a square in $\mathbb{Q}_2$
by Lemma~\ref{lem:square}.

If $v_2(c)\ge 3$, then $v_2(c^2+2n)=3$ and $m_n(c)=2v_2(c)-5\ge 1$.
By \eqref{eq:v2ak},
\begin{gather*}
	v_2(a_0) = 3n, \quad
	v_2(a_k) = v_2\left(\binom{n}{2k}\right) + m_n(c) k + 3 n \ge 3 n + 2 \quad \text{for $k \ge 1$}, \\
	v_2(b_k) = v_2\left(\binom{n-1}{2k}\right) + m_n(c) k + 2 v_2(c) + 3 (n - 1) \ge 3 n + 3  \quad \text{for $k \ge 0$}.
\end{gather*}
Since $n$ is even, $(c^2+2n)^n/2^{3n}\equiv 1\pmod{8}$
by Lemma~\ref{lem:square}.
Hence, by \eqref{eq:Dnc},
\[
	\frac{D_n(c)}{2^{3n}} \equiv
	\frac{a_0}{2^{3n}}
	\equiv \frac{(c^2 + 2 n)^n}{2^{3n}}
	\equiv 1 \pmod{4}.
\]
By \eqref{eq:DQH}, \eqref{eq:exponent} and \eqref{eq:prodk}, we have
\[
	2^{2e} \disc(H_{n;c}) \equiv 3 \cdot 1 \equiv 3 \pmod{4},
\]
where $e$ is an integer.
Therefore $\disc(H_{n;c})$ is not a square in $\mathbb{Q}_2$
by Lemma~\ref{lem:square}.


\bigskip

\noindent
\textbf{The case $n\equiv 6\pmod{8}$.}
If $v_2(c)\le -1$, then we have
\[
	\frac{D_n(c)}{2^{2(n-1)v_2(c)+1}} \equiv 1 \pmod{8}
\]
as in the case where $n\equiv 5\pmod{8}$.
By \eqref{eq:DQH}, \eqref{eq:exponent} and \eqref{eq:prodk}, we have
\[
	2^{2e} \disc(H_{n;c}) \equiv 7 \cdot 1 \equiv 7 \pmod{8},
\]
where $e$ is an integer.
Therefore $\disc(H_{n;c})$ is not a square in $\mathbb{Q}_2$
by Lemma~\ref{lem:square}.

If $v_2(c)=0$, then $v_2(c^2+2n)=0$ and $m_n(c)=1$.
By \eqref{eq:v2ak}, we have
\begin{gather*}
	v_2(a_0) = 0, \quad v_2(a_1) = 1, \quad v_2(a_2) = 2, \quad v_2(a_k)\ge 3, \\
	v_2(b_0) = 0, \quad v_2(b_1) = v_2(b_2) = 2, \quad v_2(b_k)\ge 3
\end{gather*}
for $k\ge 3$.
Since $c^2\equiv 1\pmod{8}$ and $c^2+2n\equiv 5\pmod{8}$,
by \eqref{eq:ab0},
\begin{align*}
	a_0 - b_0 &= 2 n (c^2 + 2 n)^{n-1} \equiv 2 \cdot 6 \cdot 5
	\equiv 4 \pmod{8}, \\
	a_1 &= - n (n - 1) c^2 (c^2 + 2 n)^{n-2}
	\equiv - 6 \cdot 5 \cdot 1 \cdot 1 \equiv 2 \pmod{8}.
\end{align*}
Hence, by \eqref{eq:Dnc},
\[
	D_n(c) \equiv
	a_0 - b_0 + a_1 - b_1 + a_2 - b_2
	\equiv 4 + 2 - 4 + 4 - 4
	\equiv 2 \pmod{8}.
\]
Therefore $D_n(c)/2\equiv 1\pmod{4}$.
By \eqref{eq:DQH}, \eqref{eq:exponent} and \eqref{eq:prodk}, we have
\[
	2^{2e} \disc(H_{n;c}) \equiv 7 \cdot 1 \equiv 3 \pmod{4},
\]
where $e$ is an integer.
Therefore $\disc(H_{n;c})$ is not a square in $\mathbb{Q}_2$
by Lemma~\ref{lem:square}.

If $v_2(c)=1$, then $c^2+2n\equiv 4+12\equiv 0\pmod{16}$.
Hence we have $v_2(c^2+2n)\ge 4$ and $m_n(c)\le -5$.
By \eqref{eq:v2ak},
\begin{gather*}
	v_2(a_{n/2}) = \frac{3}{2} n, \quad 
	v_2(a_k) \ge 3 k + (n - 2 k) v_2(c^2 + 2 n) \ge \frac{3}{2} n + 5, \\
	v_2(b_k) \ge 3 k + 2 + (n - 1 - 2 k) v_2(c^2 + 2 n) \ge \frac{3}{2}n + 3
\end{gather*}
for $k\le n/2-1$.
By \eqref{eq:ak} and \eqref{eq:binomial},
\[
	a_{n/2} = (-1)^{n/2} \frac{n!}{(n/2)! 0!} c^n
	= -2^{n/2} (n-1)!! c^n
\]
Since $(n-1)!!\equiv 7\pmod{8}$, and by \eqref{eq:Dnc},
\[
	\frac{D_n(c)}{2^{3n/2}} \equiv \frac{a_{n/2}}{2^{3n/2}}
	= -(n-1)!! \frac{c^n}{2^n} \equiv 1 \pmod{8}.
\]
By \eqref{eq:DQH}, \eqref{eq:exponent} and \eqref{eq:prodk}, we have
\[
	2^{2e} \disc(H_{n;c}) \equiv 7 \cdot 1 \equiv 7 \pmod{8},
\]
where $e$ is an integer.
Therefore $\disc(H_{n;c})$ is not a square in $\mathbb{Q}_2$
by Lemma~\ref{lem:square}.

If $v_2(c)\ge 2$, then $v_2(c^2+2n)=2$ and $m_n(c)=2v_2(c)-3\ge 1$.
By \eqref{eq:v2ak},
\begin{gather*}
v_2(a_0)=2n,\; v_2(a_k)\ge 2n+1 \ \text{ for $k\ge 1$},\quad v_2(b_k)\ge 2n+2 \ \text{ for $k\ge 0$}.
\end{gather*}
Hence $v_2(D_n(c))=v_2(a_0)=2n$.
Since $n\equiv 6\pmod{8}$,
we have $v_2(\disc(H_{n;c}))\equiv 1\pmod{2}$
by \eqref{eq:DQH}, \eqref{eq:exponent}, \eqref{eq:prodk}.
So $\disc(H_{n;c})$ is not a square in $\mathbb{Q}_2$
by Lemma~\ref{lem:square}.
\end{proof}

We now translate Theorem~\ref{thm:not_square} in terms of rational points on curves.
Let
\begin{equation}
\label{eq:curve}
f_r(c)=\disc(H_{r+1;c})
\end{equation}
Then $f_r(c)$ is a polynomial in $c$ of degree $2r$ with integer coefficients.
Let $C_r$ be the hyperelliptic curve defined by $y^2=f_r(x)$.

\begin{theorem} \label{thm:hypell}
The curve $C_r$ has no $\mathbb{Q}_2$-rational points if and only if $r \equiv 2,3,4,5,6\pmod{8}$.
\end{theorem}
\begin{proof}[Proof of Theorem~\ref{thm:hypell}]
Assume that $r\equiv 2,3,4,5,6\pmod{8}$.
By Theorem~\ref{thm:not_square}, it is sufficient to prove that
the points at infinity of $C_r$ are not $\mathbb{Q}_2$-rational.
By the proof of Theorem~\ref{thm:not_square},
the leading coefficient of $f_r(x)$ is equal to
\begin{equation}\label{eq:LC}
	2^{(r+1)(3r-2)/2+1} \prod_{k=1}^r k^k.
\end{equation}
It is not a square in $\mathbb{Q}_2$ by Lemma~\ref{lem:square}, \eqref{eq:exponent} and \eqref{eq:prodk}.
Therefore the points at infinity of $C_r$ are not $\mathbb{Q}_2$-rational.

Assume that $r\equiv 0,7\pmod{8}$.
By Remark~\ref{rem:Hilbert},
\[
	f_r(0) = \disc(H_{r+1}) = 2^{3r(r+1)/2} \prod_{k=1}^{r+1} k^k.
\]
If $r\equiv 0,7\pmod{8}$, then $3r(r+1)/2\equiv 0\pmod{2}$.
By \eqref{eq:prodk}, we have
\[
	2^{2e} \prod_{k=1}^{r+1} k^k \equiv 1 \pmod{8},
\]
where $e$ is an integer. Hence $f_r(0)$ is a square in $\mathbb{Q}_2$
by Lemma~\ref{lem:square}.
Therefore $C_r$ has a $\mathbb{Q}_2$-rational point.

Finally, assume that $r\equiv 1\pmod{8}$.
Then the leading coefficient of $f_r(x)$ is a square in 
$\mathbb{Q}_2$ by Lemma~\ref{lem:square}, \eqref{eq:prodk} and \eqref{eq:LC}.
Therefore the points at infinity of $C_r$ are $\mathbb{Q}_2$-rational.
\end{proof}

\begin{remark}
In fact, if $r\equiv 1\pmod{8}$, then $f_r(x)$ is a square
in $\mathbb{Q}_2$ when $v_2(x)$ is sufficiently small.
Therefore $C_r$ has a $\mathbb{Q}_2$-rational point in the affine part.
\end{remark}

\section{Conclusion and further remarks} \label{sec:conclusion}

We have derived explicit formulas for the resultants and discriminants of
all classical quasi-orthogonal polynomials of order one,
as a full generalization of the results of Dilcher and Stolarsky~\cite{DS05} and Gishe and Ismail~\cite{G06,GI08}.
Theorem~\ref{thm:res-general} for resultants is a rather general result, whereas it is not so easy to establish substantial generalizations of Theorem~\ref{thm:D-general} since our proof of Theorem~\ref{thm:D-general} employs the derivative properties of classical quasi-orthogonal polynomials. This is an interesting question, which is left for future work.

We have also dealt with Hausdorff-type equations
and created the relationship to
quasi-Hermite polynomials and quadrature formulas for Gaussian integration.
We have then proved a necessary and sufficient condition for
the hyperelliptic curve $C_r: y^2 = {\rm disc}(H_{r+1;x})$ to have $\qq_2$-rational points.
This not only provides a nonexistence theorem for solutions of Hausdorff-type equations,
but also gives us opportunities to use discriminants in the study of quadrature formulas and quasi-Hermite polynomials.

The hyperelliptic curve $C_r$ may possibly have $\qq_p$-rational points for prime numbers $p \ge 3$.
For example by using the function IsLocallySolvable in Magma~\cite{Magma}, we have examined $r \le 40$ and $p$ such that the curve $C_r$ has no $\qq_p$-rational points; see Table~\ref{tbl:prime}.
Accordingly, by the same argument as in Corollary~\ref{cor:non-exist},
the equations (\ref{eq:haus-eq-1}) for $(m,n) = (r+1,2r)$, $r \le 40$, have no rational solutions.
To improve Theorem~\ref{thm:not_square} is again left for future work.

\begin{table}[htbt]
    \begin{center}
    \caption{The prime numbers $p$ such that $C_r(\qq_p) = \emptyset$.}
    \label{tbl:prime}
\begin{tabular}{ll|ll|ll|ll}
$r$  & $p$   & $r$  & $p$         & $r$ & $p$             & $r$ & $p$     \\
\hline
1    &       &  11  & 2,3         & 21  & 2,3,11,13       & 31  & 11,31                   \\
2    & 2,3   &  12  & 2,5,7       & 22  & 2,11,13,17,19   & 32  & 23,31                   \\
3    & 2     &  13  & 2,5,7,11,13 & 23  & 3,17,23         & 33  & 3,17,23,29              \\
4    & 2     &  14  & 2,7,11,13   & 24  & 13,23           & 34  & 2,13,19,23,29,31        \\
5    & 2,3,5 &  15  & 7           & 25  & 11,19,23        & 35  & 2,5,7,13,23,29,31       \\
6    & 2,3,5 &  16  & 11          & 26  & 2,13            & 36  & 2,5,7,13,17,19,23,29,31 \\
7    & 7     &  17  & 7,11        & 27  & 2,11            & 37  & 2,13,19,23,29,37        \\
8    & 5,7   &  18  & 2,3,11      & 28  & 2,7,11,17,19,23 & 38  & 2,3,5,7,19,23,37        \\
9    & 3     &  19  & 2,17        & 29  & 2,11,29         & 39  & 17,37                   \\
10   & 2,5,7 &  20  & 2,5,11,13   & 30  & 2,3,17,19,23,29 & 40  & 5,7,17,31               \\
\end{tabular}
\end{center}
\end{table}

It may be also interesting to consider analogues of Theorem~\ref{thm:not_square} for
other classical quasi-orthogonal polynomials.
By Corollary~\ref{cor:DQC}, we have, for example
\[
   {\rm disc}(U_{r+1}) = \lim_{c \rightarrow 0} {\rm disc}(U_{r+1;c}) = 2^{(r+1)^2} (r+2)^{r-1}.
\]
This is a square in the rationals if $r$ is odd.
Therefore, for any odd integer $r$ and any prime number $p$,
the hyperelliptic curve $C_r$ has a $\qq_p$-rational point,
which does not give informations on solutions of (\ref{eq:haus-eq-1}).
Another interesting case will be the Legendre polynomials which
correspond to the integration $\frac{1}{2} \int_{-1}^1 \ dx$.
In this case, by a classical result of Holt~\cite{Ho12},
we see that there exist no rational solutions of  (\ref{eq:haus-eq-1}) for $n=2m-1$.
We are again naturally interested in the case when $n=2m-2$.

\section*{Acknowledgement}
The starting point of the study of Hausdorff-type equations was a private email conversation
the first author had with Bruce Reznick, who kindly told Hausdorff's work, more than four years ago.
The first author would like to express his sincere thanks and appreciation to him.

\appendix
\section{The classical orthogonal polynomials and some basic properties}
\label{appendix:COP}

We here describe some basic properties on Jacobi polynomials, Laguerre polynomials,
Hermite polynomials, which are used in the proof of our results.

\subsection{Jacobi polynomials}\label{appendix:JP}

The following informations can be found in~\cite[Chapter IV]{Sze}.

\bigskip

\noindent
\textbf{Closed form}
\begin{equation}\label{eq:J-explicit}
	P^{(\alpha,\beta)}_n(x)
     = \sum_{m=0}^n \binom{n+\alpha}{m} \binom{n+\beta}{n-m} \Big(\frac{x-1}{2} \Big)^{n-m} \Big( \frac{x+1}{2} \Big)^m.
\end{equation}

\noindent
\textbf{Three-term relation} 
\begin{multline}\label{eq:J-rec}
	2n (n + \alpha + \beta) (2n + \alpha + \beta - 2) P^{(\alpha,\beta)}_n(x) \\
	= (2n + \alpha + \beta - 1) \left( (2n + \alpha + \beta)
	(2n + \alpha + \beta - 2) x + \alpha^2 - \beta^2 \right)
	P^{(\alpha,\beta)}_{n-1}(x) \\
	{} - 2 (n + \alpha - 1) (n + \beta - 1) (2n + \alpha + \beta)
	P^{(\alpha,\beta)}_{n-2}(x).
\end{multline}

\noindent
\textbf{Derivative formulas} 
\begin{multline}\label{eq:DEJ1a}
	(2n + \alpha + \beta) (1 - x^2) \frac{d}{dx}P^{(\alpha,\beta)}_n(x) \\
	= -n \left( (2n + \alpha + \beta) x + \beta - \alpha \right)
	P^{(\alpha,\beta)}_n(x)
	+ 2 (n + \alpha) (n + \beta) P^{(\alpha,\beta)}_{n-1}(x),
\end{multline}
\begin{multline}\label{eq:DEJ1b}
	(2n + \alpha + \beta + 2) (1 - x^2) \frac{d}{dx}P^{(\alpha,\beta)}_n(x) \\
	= (n + \alpha + \beta + 1)
	\left( (2n + \alpha + \beta + 2) x + \alpha - \beta \right)
	P^{(\alpha,\beta)}_n(x) \\
	{} - 2 (n + 1) (n + \alpha + \beta + 1) P^{(\alpha,\beta)}_{n+1}(x),
\end{multline}

\subsection{Laguerre polynomials}\label{appendix:LP}

The following informations can be found in \cite[\S~5.1]{Sze}.

\bigskip

\noindent
\textbf{Closed form}
\begin{equation}\label{eq:L-explicit}
	L_n^{(\alpha)}(x) = \sum_{k=0}^n
	\binom{n+\alpha}{n-k} \frac{(-x)^k}{k!}.
\end{equation}

\noindent
\textbf{Three-term relation} 
\begin{equation}\label{eq:L-rec}
	n L_n^{(\alpha)}(x) = (-x + 2 n + \alpha - 1) L_{n-1}^{(\alpha)}(x)
	- (n + \alpha - 1) L_{n-2}^{(\alpha)}(x).
\end{equation}

\noindent
\textbf{Derivative formulas}
\begin{equation}\label{eq:DEL1}
	\frac{d}{dx} L_n^{(\alpha)}(x) = x^{-1}
	\left(n L_n^{(\alpha)}(x) - (n + \alpha) L_{n-1}^{(\alpha)}(x)\right).
\end{equation}

\subsection{Hermite polynomials}\label{appendix:HP}

The following informations can be found in \cite[\S~5.5]{Sze}.

\bigskip

\noindent
\textbf{Closed form}
\begin{equation}\label{eq:H-explicit}
	H_n(x) = \sum_{k=0}^{\lfloor n/2\rfloor}
	(-1)^k \frac{n!}{k! (n-2k)!} (2 x)^{n-2k}.
\end{equation}

\noindent
\textbf{Three-term relation} 
\begin{equation}\label{eq:H-rec}
	H_n(x) - 2 x H_{n-1}(x) + 2 (n - 1) H_{n-2}(x) = 0.
\end{equation}

\noindent
\textbf{Derivative formulas} 
\begin{equation}\label{eq:DEH1}
	H'_n(x) = 2 n H_{n-1}(x).
\end{equation}


\end{document}